\documentclass[11pt]{article}

\usepackage{amsmath, amssymb, amsthm, color}
\usepackage{url}
\usepackage{graphicx}
\usepackage{subcaption}
\usepackage{ulem}
\usepackage{tikz}
\usepackage{pgfplots}
\usepackage{algorithm}
\usepackage{algorithmicx}
\usepackage{algpseudocode}
\usepackage{enumitem}
\usepackage{hyperref}
\usepackage{booktabs}
\RequirePackage[capitalize,nameinlink]{cleveref}[0.19]
\pgfplotsset{compat=newest}
\usetikzlibrary{shapes,decorations}
\usetikzlibrary{fit}

\newtheorem{definition}{Definition}

\theoremstyle{thmstyletwo}%
\newtheorem{example}{Example}[section]
\newtheorem{remark}{Remark}[section]%
\newtheorem{corollary}{Corollary}[section]
\newtheorem{proposition}{Proposition}[section]
\newtheorem{lemma}{Lemma}[section]
\newtheorem{theorem}{Theorem}[section]

\theoremstyle{thmstylethree}%
\DeclareRobustCommand{\R}{\mathbb R}

\newcommand{\setE}{\mathbb{E}}

\newcommand{\setN}{\mathbb{N}}

\newcommand{\setX}{\mathbb{X}}

\newcommand{\sB}{\sf B}
\newcommand{\bsetX}{\overline{\mathbb{X}}}
\newcommand{\be}{\begin{equation}}
\newcommand{\ee}{\end{equation}}

\newcommand{\prox}{{\mbox{\rm Prox}}}

\newcommand{\dist}{\operatorname{dist}}
\newcommand{\sign}{\operatorname{sign}}
\newcommand{\proj}{\operatorname{Proj}}

\newcommand\domain[1]{\mathrm{dom\\}(#1)}
\newcommand{\N}{\mathbb{N}}

\begin{document}


\title{ A Unified Analysis on the Subgradient Upper Bounds for the Subgradient Methods Minimizing Composite Nonconvex, Nonsmooth and Non-Lipschitz Functions }


\author{
Daoli Zhu\thanks {Antai College of Economics and Management, Shanghai Jiao Tong University, 200030 Shanghai, China
({\tt dlzhu@sjtu.edu.cn}); and School of Data Science, The Chinese University of Hong Kong, Shenzhen, Shenzhen 518172, China}
\and
Lei Zhao\thanks {Institute of Translational Medicine and National Center for Translational Medicine, Shanghai
Jiao Tong University, 200030 Shanghai, China ({\tt l.zhao@sjtu.edu.cn}); and Xiangfu Laboratory, Jiashan 314000, China.}
\and
Shuzhong Zhang\thanks {Department of Industrial and Systems Engineering, University of Minnesota, Minneapolis, MN 55455, USA
({\tt zhangs@umn.edu})}}

\maketitle

\begin{abstract}%
This paper presents a unified analysis for the proximal subgradient method (Prox-SubGrad) type approach to minimize an overall objective of $f(x)+r(x)$, subject to convex constraints, where both $f$ and $r$ are weakly convex, nonsmooth, and non-Lipschitz.
Leveraging on the properties of the Moreau envelope of weakly convex functions, we are able to relate error-bound conditions, the growth conditions of the subgradients of the objective, and the behavior of the proximal subgradient iterates on some remarkably broad classes of objective functions. Various existing as well as new bounding conditions are studied, leading to novel iteration complexity results. The terrain of our exploration expands to stochastic proximal subgradient algorithms.
\end{abstract}

\vspace{0.5cm}

{\bf Keywords:} nonsmooth optimization, non-Lipschitz condition, proximal subgradient, stochastic subgradient, iteration complexity analysis.

{\bf MSC Classification:} 90C33, 65K15, 90C46, 49K35, 90C26.


\section{Introduction}\label{sec:introduction}
In this paper, we consider the following composite nonconvex and nonsmooth optimization problem:
\begin{equation}\label{eq:opt problem_prox}
\min_{x\in \setX} \ f(x)+r(x)
\end{equation}
where $f:\setX\rightarrow\R$ and $r:\setX\rightarrow\R$ are assumed to be proper and lower-semicontinuous, and $\setX$ is a closed convex subset of $\R^d$. Later in our discussion, we denote the overall objective as $\varphi(x):=f(x)+r(x)$.
Throughout this paper, we also use $\setX^*$ to denote the set of optimal solutions of~\eqref{eq:opt problem_prox}, which we assume to be non-empty. 
For any $x^*\in\setX^*$, we denote $\varphi^*=\varphi(x^*)$.
The set of critical points of problem \eqref{eq:opt problem_prox}, on the other hand, is denoted by $\overline \setX$. Let $\mathcal{I}_{\setX}(\cdot)$ be the indicator function of the set $\setX$.
Suppose that the functions $f$ and $r$ are weakly convex, which we shall define shortly. 
As the weakly convex functions are equipped with subgradients -- similar as the convex functions,
we propose to solve
problem~\eqref{eq:opt problem_prox} via the following proximal subgradient method (\Cref{alg:Prox-SubGrad}).

\begin{algorithm}[ht]
	\caption{(Prox-SubGrad). Proximal Subgradient Method for Solving \eqref{eq:opt problem_prox}}
	{\bf Initialization:}  $x^0$ and $\alpha_0$;
	\begin{algorithmic}[1]
		\item[1:] {\bf for} {$k=0,1,\ldots$} {\bf do}
		\item[2:] $\quad$Compute a subgradient $g(x^k)\in\partial f(x^k)$;
		\item[3:] $\quad$Update the step size $\alpha_{k}$ according to a certain rule;
		\item[4:] $\quad$Update  $x^{k+1} := \arg\min\limits_{x\in\setX}\,\, \langle g(x^k),x\rangle+r(x)+\frac{1}{2\alpha_k}\|x-x^k\|^2$. 
		\item[5:] {\bf end for}
	\end{algorithmic}
	\label{alg:Prox-SubGrad}
\end{algorithm}

In the above description of~\Cref{alg:Prox-SubGrad}, $\partial f(x)$ denotes the set of the subgradients of $f$ at $x$.

In the literature for analyzing the convergence of Prox-SubGrad, or the projected subgradient algorithm (Proj-SubGrad) if $r=0$, or simply the subgradient algorithm if $r=0$ and $\setX=\R^d$, a standard assumption is often made, which is that the objective function $\varphi=f+r$ is Lipschitz continuous. Equivalently put, it implies that $\|g(x)+h(x)\|\leq L$,   $\forall x\in\setX$ with $g(x)\in\partial f(x)$ and $h(x)\in\partial r(x)$, where $L$ is a positive constant. There is a vast body of literature on the topic, since the study on the subgradient algorithms has gained enormous momentum due to its applications in the machine learning research. For more readings on the topic, we refer the interested reader to, e.g., \cite{shor2012minimization,nesterov2014subgradient,dima2018subgradient}. The Lipschitz continuity assumption is often satisfied in practice. However, there are key applications where that assumption is violated, especially in the settings when the objective is nonsmooth, non-convex, and/or the level set is unbounded. Let us illustrate this by considering the following two application scenarios.

\subsection{Motivating examples}\label{subsec:motivation}
{\tt Application 1 (real-valued robust phase retrieval problem).}  Phase retrieval is a common computational problem with applications in diverse areas such as physics, imaging science, X-ray crystallography, and signal processing~\cite{fienup1982phase}. The (real-valued) robust phase retrieval problem amounts to solving \cite{RobustPhase1}
\be\label{eq:pr}
\min_{x\in\R^d}\ \varphi(x) := \frac{1}{n}\|(Ax)^{\circ2}-b^{\circ2}\|_1+R(x),
\ee
where $A=\left(a_1,...,a_n\right)^{\top}\in\R^{n\times d}$ is the data matrix, $b=\left(b_1,...,b_n\right)^{\top}\in\R^n$ is a vector, and $\circ2$ is the Hadamard power. Note that compared to the optimization problem discussed in \cite{RobustPhase1},~\eqref{eq:pr} has an extra regularizer $R$ (e.g.\ $R(x)=\frac{p}{2}\|x\|^2$), which is known to reduce the variance of an estimator~\cite{ma2019optimizationbased}. Let $\varphi=f+r$ with $f(x)=\frac{1}{n}\|(Ax)^{\circ2}-b^{\circ2}\|_1$ and $r(x)=R(x)=\frac{p}{2}\|x\|^2$, where the subgradient is of the form $\frac{1}{n}A^{\top}\left(Ax\circ\sign((Ax)^{\circ2}-b^{\circ2})\right)+px$ with $\circ$ being the Hadamard product. It is easy to see that this subgradient can grow to infinity, but at a rate no more than linear. Therefore, both $f$ and $r$ in this case are non-Lipschitz, and $f$ is nonconvex.

{\tt Application 2 (blind deconvolution and biconvex compressive sensing).} The problem of blind deconvolution seeks to recover a pair of vectors in two low-dimensional structured spaces from their pairwise convolution.  This problem occurs in a number of fields, such as astronomy and computer vision~\cite{chan1998total,levin2011understanding}.  For simplicity we focus on the real-valued case. One 
formulation of the problem reads
\be\label{eq:bdbc}
\min_{x,y}\frac{1}{n}\sum_{i=1}^n|\langle a_i,x\rangle\langle b_i, 
y \rangle-c_i|,
\ee
where $a_i$ and $b_i$ are given vectors, and $c_i$ are the convolution measurements.  More broadly,  problems of this form fall within the area of biconvex compressive sensing~\cite{ling2015self}.  Similarly to the previous example, the use of the $\ell_1$-penalty on the residuals yields strong recovery and stability guarantees under statistical assumptions. Since the $i$-th component of the $x$ part of the subgradient is of the form $\frac{1}{n}\sum_{i=1}^n \mbox{sign}(\langle a_i,x\rangle \langle b_i, y \rangle - c_i) \langle b_i, y \rangle a_i$, and the $i$th component of the $y$ part of the subgradient is of the form $\frac{1}{n}\sum_{i=1}^n \mbox{sign}(\langle a_i,x\rangle
\langle b_i, y \rangle - c_i) \langle a_i,x\rangle b_i$. It is easy to see that subgradient can grow, but at most linearly. Therefore, the problem is non-convex and non-Lipschitz.

Motivated by such examples, our goal in this paper is to develop an encompassing convergence analysis for Prox-SubGrad without resorting to Lipschitz continuity assumptions.
\subsection{Related literature}
To avoid
issues caused by being non-Lipschitz, early in the development of subgradient methods Shor~\cite{shor2012minimization}, Cohen and Zhu~\cite{CohenZ}, Nesterov~\cite{nesterov03} studied normalized subgradient methods (meaning that \Cref{alg:Prox-SubGrad} with normalized stepsize, i.e., $\alpha_k=\beta_k/\|g(x^k)\|$) for problem~\eqref{eq:opt problem_prox} with $r=0$, which guarantee 
convergence for the convex case. This scenario can be seen as a specialization of scenario (B$_1$) in~\Cref{sec:subgrad_upp}. Moreover, Cohen and Zhu~\cite{CohenZ} proposed some Prox-SubGrad methods assuming the subgradient $\|g(x)\|$ ($r=0$) or $\|g(x)+h(x)\|$ ($r\neq0$) satisfies the linearly bounded condition (see (B$_2$) of~\Cref{sec:subgrad_upp}) to guarantee convergence for non-Lipschitz convex problems. For this case, the step size design is not divided by $\|g(x^k)\|$. The linearly bounded stochastic subgradient condition on every realization of subgradient (i.e., $\forall x$, $\forall\xi$, $\forall g(x,\xi)\in\partial f(x,\xi)$, $\forall h(x)\in\partial r(x)$, $\|g(x,\xi)+h(x)\|^2\leq c_1\|x\|^2+c_2$ with $c_1,c_2\geq0$) is used in Culioli and Cohen~\cite{culioli1990} to solve convex stochastic optimization.

More recently, Li {\it et al.}~\cite{Li2023subgrad} extended the complexity results of normalized subgradient methods (\Cref{alg:Prox-SubGrad} with normalized stepsize, i.e., $\alpha_k=\beta_k/\|g(x^k)\|$ and $r=0$) to convex and weakly convex minimization without assuming Lipschitz continuity. Their complexity results were further extended to the truncated subgradients, the stochastic subgradients (as well as its clipped version), and the proximal subgradient methods for non-Lipschitz continuous functions, which can also be seen as a specialization of scenario (B$_1$) in~\Cref{sec:subgrad_upp}, since by setting $r=0$, condition (B$_1$) is satisfied in their setting and so our result generalizes the result in \cite{Li2023subgrad}.

Along a different line, Renegar~\cite{renegar2016} developed a framework to convert an originally non-Lipschitz continuous problem ($r=0$) into an equivalent Lipschitz continuous problem in a slightly lifted space. Grimmer~\cite{grimmer2018} extended the work of Renegar, leading to the so-called radial subgradient methods.
Grimmer~\cite{grimmer2019} studied the normalized subgradient scheme ($\alpha_k=\beta_k/\|g(x^k)\|$ and $r=0$) under a function upper bound condition for deterministic nonsmooth convex optimization, and obtained an extended convergence rate for this case. 
Moreover, the author proposed another expected quadratic growth-type subgradient upper bound (see (B$_{\mbox{\tiny {G}}}$) in~\Cref{sec:subgrad_upp}) and showed an $O(1/\sqrt{T})$ convergence rate for the stochastic nonsmooth convex problems without Lipschitz continuity. For the weakly convex case with Lipschitz continuous $f$ ((B$_{\mbox{\tiny {D-D}}}$) in~\Cref{sec:subgrad_upp}), Davis and Drusvyatskiy~\cite{davis2019stochastic} proposed a stochastic model-based minimization approach and showed that a stationarity measure would converge to zero.

Lu \cite{lu2019} introduced a notion of relative continuity to impose some relaxed bound on the subgradients by using a Bregman distance,
which is embedded
in a mirror descent-type algorithmic design. This approach differs from a classical subgradient method.  Similar ideas were also used in \cite{zhou2020} for online convex optimization.

Linearly bounded expected stochastic subgradients and $\dist(x, X^*)$-expected stochastic subgradient upper bounds ((B$_{\mbox{\tiny {A-D}}}$) in~\Cref{sec:subgrad_upp}) were used by Asi and Duchi~\cite{asi2019importance} to analyze the convergence of stochastic subgradient method to solve weakly convex stochastic optimization problems. Additionally, Asi and Duchi~\cite{asi2019importance} proposed a general subgradient upper bound $\|g(x)+h(x)\|\leq G_{\mbox{\tiny\rm big}}\left(\|x-x^*\|\right)$, for a given $x^*\in\setX^*$, with $G_{\mbox{\tiny\rm big}}(\cdot)$ being a certain nondecreasing nonnegative function. Via this subgradient upper bound, Asi and Duchi~\cite{asi2019importance} studied convex and weakly convex stochastic optimization problem, obtaining  
convergence results under some additional 
conditions.

\subsection{Main contributions and outline of the paper}
In this paper, we introduce a number of 
new subgradient (or expected stochastic subgradient) upper bounding conditions and establish a unified convergence analysis for Prox-SubGrad (Sto-SubGrad) under both the deterministic and stochastic nonsmooth weakly convex optimization settings without Lipschitz continuity assumptions. Using these subgradient upper bounding conditions, we establish a Moreau envelope uniform recursive relationship for the weakly convex setting. The afore-mentioned scheme simplifies and unifies the analysis 
leading to an iteration complexity result 
for Prox-SubGrad (Sto-SubGrad) without the Lipschitz continuity condition. Some new convergence results are also provided. To both deterministic and stochastic subgradient methods on weakly convex optimization problems without the Lipschitz condition, we obtain an $O(1/\sqrt{T})$ convergence rate, which improves to $O(1/{T})$ under the KL 
condition. We also analyze the normalized subgradient scheme ($\alpha_k=\beta_k/\|g(x^k)\|$ and $r=0$) under a
general subgradient upper bound condition for deterministic nonsmooth weakly convex optimization and obtain an extended convergence rate for this case. We also provide a linear convergence analysis under an additional quadratic growth condition to both deterministic and stochastic subgradient methods on weakly convex optimization problems without the Lipschitz condition.

The paper is organized as follows. In~\Cref{sec:pre}, we shall introduce the notations and some preparations 
to lay the ground for our analysis.
In~\Cref{sec:subgrad_upp}, we shall introduce various subgradient upper bounding conditions for the function $\varphi=f+r$ which leads to establishing the uniform recursion of algorithm and the convergence analysis. The relationships between these subgradient upper bounding conditions are presented. Furthermore, in~\Cref{sec:wcvx_wcvx} we discuss the convergence analysis of Prox-SubGrad when $\varphi$ is weakly convex and non-Lipschitz. Then we establish the Moreau envelope 
recursive relations with the subgradient upper bounding conditions (B$_{\mbox{\tiny {D-D}}}$), (B$_1$) and (B$_3$) and (B$_4$), which leads to a rate of convergence for Prox-SubGrad to minimize nonsmooth weakly convex functions without any Lipschitz continuity assumption. By using the generalized subgradient upper bound, we obtain the extended convergence rate results for the projected subgradient method (Proj-SubGrad). We also provide the linear convergence results of Prox-SubGrad under an additional quadratic growth and/or sharpness condition. \Cref{sec:sgd} is devoted to stochastic subgradient method (Sto-SubGrad) to solve nonsmooth stochastic optimization problem. 
We propose the expect stochastic subgradient upper bounding conditions (B$_3'$) and (B$_4'$) which allow us to establish a convergence analysis for Sto-SubGrad in the same unifying fashion. By using an additional quadratic growth condition, we also obtain the linear rate of Sto-SubGrad.
In order not to disrupt too much the flow of the presentation, some of the technical proofs are relegated to the appendices.
\section{Preliminaries}\label{sec:pre}
\subsection{Convexity, weak convexity, and subdifferentials} \label{sec:cvx-wcvx-subgrad}
A function $\varphi:\setX\rightarrow\R$ is convex if for all $x,y\in\setX$ and $0\leq\eta\leq 1$, we have
$\varphi((1-\eta)x+\eta y)\leq(1-\eta)\varphi(x)+\eta\varphi(y)$.
A vector $v\in \R^d$ is called a {subgradient} of $\varphi$ at point $x$ if the subgradient inequality $\varphi(y)\geq\varphi(x)+\langle v,y-x\rangle$
holds for all $y\in\setX$. The set of all subgradients of $\varphi$ at $x$ is denoted by $\partial\varphi(x)$, and is called the (convex) subdifferential of $\varphi$ at $x$.

Now, further along that line, a function $\varphi:\setX\rightarrow\R$ is called {\it $\rho$-weakly convex}\/ if there is $\rho>0$ such that
$\theta(x) := \varphi (x) + \frac{\rho}{2}\|x\|^2$ is convex.
For such a $\rho$-weakly convex function $\varphi$,  its subdifferential is defined as (cf.~\cite[Proposition 4.6]{Vial83})
\(
\partial\varphi(x) := \partial \theta(x)-\rho x,
\)
where $\partial \theta$ is the subdifferential of the regular convex function $\theta(x)$. 
Additionally, it is well known that the $\rho$-weak convexity of $\varphi$ is equivalent to \cite[Proposition 4.8]{Vial83}
\begin{center}
$\varphi(y)\geq\varphi(x)+\langle\nu,y-x\rangle-\frac{\rho}{2}\|x-y\|^2$
\end{center}
for all $x,y\in\setX$ and $\nu \in \partial \varphi(x)$.

\subsection{ The Moreau envelope of weakly convex function}\label{subsec:me_wcvx}
Let $\varphi$ be a $\rho$-weakly convex, proper, and lower semicontinuous function on $\setX$. Then, the Moreau envelope function and the proximal mapping of $\varphi$ are defined as \cite{Rockafellar}:
\begin{align}
	&\varphi_{\lambda}(x):=\min_{y}\left\{\varphi(y)+\mathcal{I}_{\setX}(y)+\frac{1}{2\lambda}\|y-x\|^2\right\}=\min_{y\in\setX}\left\{\varphi(y)+\frac{1}{2\lambda}\|y-x\|^2\right\},\label{eq:ME}\\
	&\prox_{\lambda,\varphi}(x):=\arg\min_{y}\left\{\varphi(y)+\mathcal{I}_{\setX}(y)+\frac{1}{2\lambda}\|y-x\|^2\right\}=\arg\min_{y\in\setX}\left\{\varphi(y)+\frac{1}{2\lambda}\|y-x\|^2\right\}.\label{eq:PM}
\end{align}
Note that whenever $\lambda <\frac{1}{\rho}$, the evaluation of $\varphi_{\lambda}(x)$ for a fixed $x$ reduces to a convex optimization problem. Moreover, as we will show below, if
$\lambda <\frac{1}{\rho}$ then
$\varphi_{\lambda}$ becomes smooth and gradient Lipschitz, even though it remains nonconvex when $\varphi$ is nonconvex.

In the next proposition a series of important properties of the Moreau envelope and the proximal mapping will be presented, whose proofs will be similar to that in~\cite[Propositions 1-4] {zhudenglizhao2021}, by setting  $f=0$, $g=\varphi+\mathcal{I}_{\setX}$ and $D^k(x,y)=\frac{1}{2}\|x-y\|^2$.
\begin{proposition}\label{prop:wcvxf_me}Suppose that $\varphi$ is a $\rho$-weakly convex function. Choose  $\lambda <\frac{1}{\rho}$. For any $x\in\setX$, the following assertions hold true:
\begin{enumerate}[label=\textup{\textrm{(\alph*)}},topsep=0pt,itemsep=0ex,partopsep=0ex]
\item[{\rm (i)}] $\prox_{\lambda,\varphi}(x)$ is well defined, single-valued, and Lipschitz continuous.
\item[{\rm (ii)}] $\varphi_{\lambda}(x)\leq\varphi(x)-\frac{1-\lambda\rho}{2\lambda}\|x-\prox_{\lambda,\varphi}(x)\|^2$.
\item[{\rm (iii)}] $\lambda\dist(0,\partial\varphi(\prox_{\lambda,\varphi}(x))+\mathcal{N}_{\setX}(\prox_{\lambda,\varphi}(x)))\leq\|x-\prox_{\lambda,\varphi}(x)\|\leq\frac{\lambda}{1-\lambda\rho}\dist(0,\partial\varphi(x)+\mathcal{N}_{\setX}(x))$, where  $\mathcal{N}_{\setX}(x):= \{d\mid \langle d, x'-x\rangle\leq0, \forall x'\in\setX\}$ is the normal cone at $x$ with respect to $\setX$.
\item[{\rm (iv)}] $\nabla\varphi_{\lambda}(x)=\frac{1}{\lambda}(x-\prox_{\lambda,\varphi}(x))$ and it is Lipschitz continuous.
\item[{\rm (v)}] $x=\prox_{\lambda,\varphi}(x)$ if and only if $0\in\partial\varphi(x)+\mathcal{N}_{\setX}(x)$.
\end{enumerate}
\end{proposition}

In the following corollaries, we observe that the weakly convex function and its Moreau envelope mapping have the same set of optimal solutions and critical points. We refer the readers to~\cite{zhao2022randomized}, where the proofs for the following corollaries can be deduced:
\begin{corollary}\label{cor:me}
	Suppose that $\varphi$ is a $\rho$-weakly convex function. Let $\setX_{\lambda}^*$  denote the  set of minimizers of the problem $\min_{x\in\setX}\ \varphi_{\lambda}(x)$.  Then, the following statements hold true:
	\begin{enumerate}[label=\textup{\textrm{(\alph*)}},topsep=0pt,itemsep=0ex,partopsep=0ex]
		\item[{\rm (i)}] $\varphi_{\lambda}(x)\geq \varphi^*$ for all $ x\in\setX$.
		\item[{\rm (ii)}] $\varphi_{\lambda}(x_{\lambda}^*)=\varphi^*$ for all $x_{\lambda}^*\in\setX_{\lambda}^*$. Consequently, we have $\setX_{\lambda}^*=\setX^*$.
	\end{enumerate}
\end{corollary}
\begin{corollary}\label{cor:mecritical} Suppose that $\varphi$ is a $\rho$-weakly convex function.   Let  $\overline{\setX}_{\lambda}$ be the set of critical points of the problem  $\min_{x\in\setX}\ \varphi_{\lambda}(x)$.  Then, we have  $\overline{\setX}=\overline{\setX}_{\lambda}$.
\end{corollary}
\begin{lemma}
\label{lemma:phi_lambda_1} If $\varphi$ is weakly convex, then for any $x$ and $z\in\setX$ we have
\[
\varphi(z)-\varphi(\prox_{\lambda,\varphi}(x))\geq\frac{\|\prox_{\lambda,\varphi}(x)-x\|^2-\|z-x\|^2}{2\lambda}.
\]
\end{lemma}
\begin{proof} By the definition of $\prox_{\lambda,\varphi}(x)$, we have
\[
\begin{aligned}
\varphi(z)+\frac{1}{2\lambda}\|z-x\|^2&\geq\min_{y\in\setX} \left[ \varphi(y)+\frac{1}{2\lambda}\|y-x\|^2\right] \\
&=\varphi(\prox_{\lambda,\varphi}(x))+\frac{1}{2\lambda}\|\prox_{\lambda,\varphi}(x)-x\|^2,
\end{aligned}
\]
implying the desired result.
\end{proof}
\begin{lemma}
\label{lemma:phi_lambda_2} If $\varphi$ is weakly convex, then for any $x$ and $z\in\setX$ we have
\[
\varphi_{\lambda}(z)-\varphi_{\lambda}(x)\leq\frac{\|\prox_{\lambda,\varphi}(x)-z\|^2-\|\prox_{\lambda,\varphi}(x)-x\|^2}{2\lambda}.
\]
\end{lemma}
\begin{proof} By the definition of $\varphi_{\lambda}(z)$ and $\prox_{\lambda,\varphi}(x)$, we have
\[
\begin{aligned}
\varphi_{\lambda}(z)&=\min_{y\in\setX} \left[\varphi(y)+\frac{1}{2\lambda}\|y-z\|^2 \right] \\
&\leq\varphi(\prox_{\lambda,\varphi}(x))+\frac{1}{2\lambda}\|\prox_{\lambda,\varphi}(x)-z\|^2\\
&=\varphi_{\lambda}(x)-\frac{1}{2\lambda}\|\prox_{\lambda,\varphi}(x)-x\|^2+\frac{1}{2\lambda}\|\prox_{\lambda,\varphi}(x)-z\|^2,
\end{aligned}
\]
which is the desired result.
\end{proof}
\subsection{Uniform Kurdyka-{\/L}ojasiewicz (KL) and global KL property for $\varphi$ and $\varphi_{\lambda}$}
In general, if a function has a relevant KL type property, then it may help to improve the convergence rate of the algorithm. In our bid for 
a better convergence rate of the Prox-SubGrad algorithm for the weakly convex case (see~\Cref{cor:qg_wcvx}), we need the notion of uniform KL property with exponent $1/2$ for the 
Moreau envelope $\varphi_{\lambda}$.

\begin{definition}\label{KL properties}
 {\rm For a $\rho$-weakly convex function $\varphi:\setX\rightarrow\R$, consider the following bounded KL property for $\varphi$ and $\varphi_{\lambda}(x)=\min_{y\in\setX}\left[\varphi(y)+\frac{1}{2\lambda}\|y-x\|^2\right]$ with $\lambda<\frac{1}{\rho}$. Let $V:=\overline{\mathbb{B}}_{\epsilon}(\bar{x})\cap\{x\in\setX\mid \bar{\varphi}<\varphi(x)<\bar{\varphi}+\gamma\}$ with $\bar{x}\in\overline{\setX}$, where $\overline{\mathbb{B}}_{\epsilon}(\bar{x})=\{x\in\setX\mid\|x-\bar{x}\|\leq\epsilon\}$.
 We introduce three KL type error bound conditions regarding $\varphi$ and $\varphi_{\lambda}$ as follows:
\begin{enumerate}[label=\textup{\textrm{(\alph*)}},topsep=0pt,itemsep=0ex,partopsep=0ex]
\item $\varphi$ is said to satisfy the KL property with exponent $1/2$ at $\bar x$ if  there exists $\kappa_{\tiny{ KL}}^{\varphi}>0$ such that $\varphi(x)-
\bar{\varphi}
\leq\kappa_{\tiny{ KL}}^{\varphi} \dist^2(0,\partial \varphi(x)+\mathcal{N}_{\setX}(x))$, $\forall x\in V$.
\item $\varphi_{\lambda}$ is said to satisfy the KL property with exponent $1/2$ at $\bar x$ if there exists $\kappa_{\tiny KL}^{\varphi_{\lambda}}>0$ such that $\varphi_{\lambda}(x)-
  \bar{\varphi}
  \leq\kappa_{\tiny KL}^{\varphi_{\lambda}} \|\nabla\varphi_{\lambda}(x)\|^2$, $\forall x\in V$.
\item $\varphi_{\lambda}$ is said to satisfy the uniform KL property with exponent $1/2$ if there exists $\kappa_{\tiny uKL}^{\varphi_{\lambda}}>0$ such that $\varphi_{\lambda}(x)-\bar{\varphi}\leq\kappa_{\tiny uKL}^{\varphi_{\lambda}} \|\nabla\varphi_{\lambda}(x)\|^2$, $\forall x\in\mathfrak{B}(\bar{\varphi},\epsilon,\nu)$, where $\mathfrak{B}(\bar{\varphi},\epsilon,\nu)=\{x\in\setX\mid \dist(x,\Omega)\leq\epsilon, \bar{\varphi}<\varphi_{\lambda}(x)<\bar{\varphi}+\nu\}$ with $\Omega=\{x\in\setX\mid\varphi_{\lambda}(x)=\bar{\varphi}$ for some $\bar{x}\in\overline{\setX}$ and $\bar{\varphi}=\varphi(\bar{x})\}$.
\end{enumerate}
}
\end{definition}

The following proposition depicts the relationship among the three types of KL properties with exponent $1/2$ for the weakly convex functions as introduced in Definition~\ref{KL properties}.

\begin{proposition}[KL and Uniform KL (uKL)]\label{prop:eb}
For the three KL type conditions introduced in Definition~\ref{KL properties}, we have:
\begin{itemize}
\item[{\rm(i)}] {\rm (a)} $\Longrightarrow$ {\rm (b)}.
\item[{\rm(ii)}] Let $\Omega$ be a nonempty compact set, and suppose that $\varphi$ satisfies the KL property with exponent $1/2$ at each point in $\Omega$. Then, {\rm (c)} holds.
\end{itemize}
\end{proposition}
\begin{proof}
{\rm(i)} By taking $x=\prox_{\lambda, \varphi}(x)$ in the KL inequality of $\varphi$ and the definition of $\varphi_{\lambda}$, we have
\[
\begin{aligned}
	\varphi_{\lambda}(x)-
 \bar{\varphi}
 &=\varphi(\prox_{\lambda, \varphi}(x))-
 \bar{\varphi}
 +\frac{1}{2\lambda}\|x-\prox_{\lambda, \varphi}(x)\|^2\\
&\leq\kappa_{\tiny{ KL}}^{\varphi}\dist^2(0,\partial \varphi(\prox_{\lambda, \varphi}(x))+\mathcal{N}_{\setX}(\prox_{\lambda, \varphi}(x)))+\frac{1}{2\lambda}\|x-\prox_{\lambda, \varphi}(x)\|^2.
\end{aligned}
	\]
By (iii) of~\Cref{prop:wcvxf_me} and the fact that $\|\nabla\varphi_{\lambda}(x)\|=\frac{1}{\lambda}\|x-\prox_{\lambda, \varphi}(x)\|$, we obtain the desired result.

{\rm(ii)} See Lemma 2.2 of~\cite{TK22}.
\end{proof}

For stochastic subgradient method, 
the notion of {\it global}\/ KL property with exponent $1/2$ of the function $f$ 
becomes useful. 
The notion of global quadratic growth condition and global metric subregularity are used to study the relationship among subgradient upper bounds (cf.~\Cref{prop:upperbound3}).

\begin{definition} \label{gKL}
{\rm For the functions $\varphi$ and $\varphi_{\lambda}:\setX\rightarrow\R$ as defined in \eqref{eq:ME}, 
let us consider the following six global error bound conditions regarding $\varphi$ and $\varphi_{\lambda}$:
\begin{enumerate}
\item[{\rm (a)}] $\varphi$ is said to satisfy the global KL property (gKL) with exponent $1/2$ if there exists $\kappa_{\tiny gKL}^{\varphi}>0$ such that $\varphi(x)-\varphi^*\leq\kappa_{\tiny gKL}^{\varphi} \dist^2(0,\partial \varphi(x)+\mathcal{N}_{\setX}(x))$, $\forall x\in\setX$.
\item[{\rm (b)}] $\varphi_{\lambda}$ is said to satisfy the global KL inequality if there exists $\kappa_{\tiny gKL}^{\varphi_{\lambda}}>0$ such that $\varphi_{\lambda}(x)-\varphi^*\leq\kappa_{\tiny gKL}^{\varphi_{\lambda}}\|\nabla\varphi_{\lambda}(x)\|^2$, $\forall x\in\setX$.
\item[{\rm (c)}] $\varphi_{\lambda}$ is said to satisfy the global quadratic growth condition (gQG) if there exists $\kappa_{\tiny gQG}^{\varphi_{\lambda}}>0$ such that $
\dist^2(x,\setX^*)\leq\kappa_{\tiny gQG}^{\varphi_{\lambda}}[\varphi_{\lambda}(x)-\varphi^*]$, $\forall x\in\setX$.
\item[{\rm (d)}] $\varphi$ is said to satisfy the global quadratic growth condition if there exists $\kappa_{\tiny gQG}^{\varphi}>0$ such that $\dist^2(x,\setX^*)\leq\kappa_{\tiny gQG}^{\varphi}[\varphi(x)-\varphi^*]$, $\forall x\in\setX$.
\item[{\rm (e)}] $\partial\varphi$ is said to satisfy the global metric subregularity (gMS) if there exists $\kappa_{\tiny gMS}^{\varphi}>0$ such that $\dist(x,\overline{\setX})\leq\kappa_{\tiny gMS}^{\varphi}\dist(0,\partial\varphi(x)+\mathcal{N}_{\setX}(x))$, $\forall x\in\setX$.
\item[{\rm (f)}] $\nabla\varphi_{\lambda}$ is said to satisfy the global metric subregularity if there exists $\kappa_{\tiny gMS}^{\varphi_{\lambda}}>0$ such that $\dist(x,\overline{\setX})\leq\kappa_{\tiny gMS}^{\varphi_{\lambda}}\|\nabla\varphi_{\lambda}(x)\|$, $\forall x\in\setX$.
\end{enumerate}
}
\end{definition}
The next proposition establishes the relations regarding the global KL condition 
and other global error bounds for $\varphi$ and $\varphi_{\lambda}$.

\begin{proposition}[Global error bounds]\label{prop:global_eb}

Let $\varphi$ be $\rho$-weakly convex. For the six global KL type error bound conditions introduced in
Definition~\ref{gKL}, we have:
\begin{itemize}
\item[{\rm (i)}] {\rm(a)} $\Longrightarrow$ {\rm(b)} $\Longrightarrow$ {\rm(c)} $\Longrightarrow$ {\rm(d)};
\item[{\rm (ii)}] {\rm(a)} $\Longrightarrow$ {\rm(e)};
\item[{\rm (iii)}] {\rm(e)} $\Longleftrightarrow$ {\rm(f)}.
\end{itemize}
\end{proposition}
\begin{proof}
{\rm(i)} (a) $\Longrightarrow$ (b): Since $\varphi$ satisfies the global KL property with exponent $1/2$, by statements (iii) and (iv) of~\Cref{prop:wcvxf_me}, we have
\[
\begin{aligned}
\varphi(\prox_{\lambda,\varphi}(x))-\varphi^*&\leq\kappa_{\tiny gKL}^{\varphi}\dist^2(0,\partial \varphi(\prox_{\lambda,\varphi}(x))+\mathcal{N}_{\setX}(\prox_{\lambda, \varphi}(x)))\\
&\leq\frac{\kappa_{\tiny gKL}^{\varphi}}{\lambda^2}\|x-\prox_{\lambda,\varphi}(x)\|^2=\kappa_{\tiny gKL}^{\varphi}\|\nabla\varphi_{\lambda}(x)\|^2.
\end{aligned}
\]
Adding the equation $\frac{1}{2\lambda}\|x-\prox_{\lambda,\varphi}(x)\|^2=\frac{\lambda}{2}\|\nabla\varphi_{\lambda}(x)\|^2$ on the above inequality, we obtain
\[
\varphi_{\lambda}(x)-\varphi^*\leq \left(\kappa_{\tiny gKL}^{\varphi}+\frac{\lambda}{2}\right)\|\nabla\varphi_{\lambda}(x)\|^2,
\]
which yields that $\varphi_{\lambda}$ satisfies the KL inequality.

(b) $\Longrightarrow$ (c): See Theorem 2 of~\cite{karimi2016linear}.

(c) $\Longrightarrow$ (d): By~\Cref{cor:me} and statement (b) of~\Cref{prop:wcvxf_me}, we obtain the desired result.

{\rm(ii)} (a) $\Longrightarrow$ (e): Since (a) $\Longrightarrow$ (d) in statement (i) of this proposition, we have that
\[
\begin{aligned}
\dist^2(x,\overline{\setX})\leq\dist^2(x,\setX^*)&\overset{\rm(d)}{\leq}\kappa_{\tiny gQG}^{\varphi}[\varphi(x)-\varphi^*]\\
&\overset{\rm(a)}{\leq}\kappa_{\tiny gKL}^{\varphi}\kappa_{\tiny gQG}^{\varphi}\dist^2(0,\partial\varphi(x)+\mathcal{N}_{\setX}(x)),\;\forall x\in\setX,
\end{aligned}
\]
which shows $\partial\varphi$ is global metric subregularity with $\kappa_{\tiny gMS}^{\varphi}=\sqrt{\kappa_{\tiny gKL}^{\varphi}\kappa_{\tiny gQG}^{\varphi}}$.

{\rm(iii)} (e) $\Longleftrightarrow$ (f): See~\cite[Lemma 4.2, Corollary 4.1]{zhao2022randomized}.
\end{proof}

\begin{remark} \label{rem2.1} 
{\rm Proposition~\ref{prop:global_eb} shows that the global KL condition implies the uniform KL condition.}
\end{remark}
\begin{remark} {\rm From the definition of global KL property for $\varphi$, every stationary point $x$ ($0\in\partial\varphi(x)+\mathcal{N}_{\setX}(x)$) implies it is a global minimizer, i.e., $\varphi$ is a nonsmooth invex function~\cite{invex}.}
\end{remark}

\section{Subgradient upper bounding conditions}\label{sec:subgrad_upp}
In this section, we shall introduce a number of upper bounding conditions for the 
subgradient in order 
to guarantee the convergence of Prox-SubGrad without assuming Lipschitz continuity of the objective function $\varphi(x)$. In the following, three groups of upper bounding conditions are presented. The first group is partial Lipschitz subgradient upper bounds (see~\Cref{def:1}).
In particular, the 
subgradient upper bound (B$_{\mbox{\tiny {D-D}}}$) was used in~\cite{davis2019stochastic}, and (B$_1$) was used in~\cite{shor2012minimization, CohenZ, Li2023subgrad}.
\begin{definition}[Partial Lipschitz subgradient upper bounds]\label{def:1} {\ }
{\rm
\begin{itemize}
\item[{\rm (B$_{\mbox{\tiny {D-D}}}$).}] {\rm($f$-Lipschitz SUB)} $\partial\varphi$ is said to have $f$-Lipschitz subgradient upper bound if for any $x\in\setX$, $\forall g(x)\in \partial f(x)$, we have $\|g(x)\|\leq L_f$ ($L_f$ being a Lipschitz constant for $f$).
\item[{\rm (B$_1$).}] {\rm($r$-Lipschitz SUB)} $\partial\varphi$ is said to have $r$-Lipschitz subgradient upper bound if for any $x\in\setX$, $\forall h(x)\in \partial r(x)$, we have $\|h(x)\|\leq L_r$ ($L_r$ being a Lipschitz constant for $r$).
\end{itemize}
}
\end{definition}
The following basic subgradient upper bound is
also useful in our convergence analysis. 
Remark that the subgradient upper bound (B$_2$) was used in~\cite{CohenZ}.

\begin{definition}[Non-Lipschitz subgradient upper bounds]\label{def:2} {\ }
{\rm
\begin{itemize}
\item[{\rm (B$_2$).}] {\rm(Linear bounded SUB)} $\partial\varphi$ is said to have a linear upper bound for its subgradient if there are $c_1,\, c_2\geq0$ such that for any $x\in\setX$ and $\forall g(x)\in \partial f(x)$ and $\forall h(x)\in \partial r(x)$ it holds that $\|g(x)+h(x)\|\leq c_1\|x\|+c_2$.
\end{itemize}
}
\end{definition}
The following upper bounds on the subgradients of the {\it weak}\/ convex function turn out to be very useful as well:

\begin{definition} [New non-Lipschitz subgradient upper bounds]\label{def:3} {\ }

{\rm
\begin{itemize}
\item[{\rm (B$_3$).}] {\rm($[\varphi_{\lambda}(x)-\varphi^*]$-SUB)} $\partial\varphi$ is said to have $[\varphi_{\lambda}(x)-\varphi^*]$-subgradient upper bound if there are $c_3,\,  c_4\geq0$ such that for any $x\in\setX$ and $\forall g(x)\in \partial f(x)$ and $\forall h(x)\in \partial r(x)$ it holds that $\|g(x)+h(x)\|\leq\sqrt{c_3[\varphi_{\lambda}(x)-\varphi^*]+c_4}$.
\item[{\rm (B$_4$).}] {\rm($\|\nabla\varphi_{\lambda}(x)\|$-SUB)} $\partial\varphi$ is said to have $\|\nabla\varphi_{\lambda}(x)\|$-subgradient upper bound if there are $c_5,\, c_{6}\geq0$ such that for any $x\in\setX$ and $\forall g(x)\in \partial f(x)$ and $\forall h(x)\in \partial r(x)$ it holds that $\|g(x)+h(x)\|\leq\sqrt{c_5\|\nabla\varphi_{\lambda}(x)\|^2+c_{6}}$.
\item[{\rm (B$_5$).}]{\rm($\|\nabla\varphi_{\lambda}(x)\|$-exponent} SUB) $\partial\varphi$ is said to have $\|\nabla\varphi_{\lambda}(x)\|$-exponent subgradient upper bound if there are $c_{7}\geq0$ and $\theta\in[1,2)$ such that for any $x\in\setX$ and $\forall g(x)\in \partial f(x)$ it holds that $\|g(x)\|\leq c_7\|\nabla\varphi_{\lambda}(x)\|^{\theta}$.
\item[{\rm (B$_6$).}]{\rm($\mathcal{T}_{\gamma}$-SUB)} The norm of any subgradient is bounded on a neighborhood of $\setX^*$,  $\mathcal{T}_{\gamma}:=\{x\in\setX\mid \dist(x,\setX^*)<\gamma\}$ with some $\gamma>0$. That is, there exists $c_8>0$ such that
\(
\|g(x)+h(x)\|\leq c_8
\)
holds for all $x\in\mathcal{T}_{\gamma}$ and $\forall g(x)\in\partial f(x)$ and $h(x)\in\partial r(x)$.
\end{itemize}
}
\end{definition}

\begin{remark}
{\rm We noted that {\rm (B$_6$)} does not mean the sets $\setX^*$ and $\mathcal{T}_{\gamma}$ are both bounded.}

\begin{example}[Example that satisfies {\rm (B$_6$)} with unbounded $\setX^*$ and $\mathcal{T}_{\gamma}$]\label{exp:0}
{\rm
\[
\min_{x\in\R}\varphi(x)=\left\{\begin{array}{ll}
3(x+1)^2, &x\le -1.5 \\
-(x-k)^2-4(x-k)-3, &-1.5+k<x\leq-1+k, k=0,1,2,...\\
-(x-k)^2+1, &-1+k<x\leq-0.5+k, k=0,1,2,...
\end{array}\right.
\]
It is easy to check that $\setX^*=\{-1,0,1,...\}$ and $f^*=0$. The sets $\setX^*$ and $\mathcal{T}_{\gamma}$ are both unbounded, and $\varphi$ in this example satisfies {\rm (B$_6$)}.}

\begin{figure}
\centering
\begin{center}
\scriptsize
		\tikzstyle{format}=[rectangle,draw,thin,fill=white]
		\tikzstyle{test}=[diamond,aspect=2,draw,thin]
		\tikzstyle{point}=[coordinate,on grid,]
\begin{tikzpicture}[>=stealth]
 \draw[->](-3,0)--(3,0)node[below]{$x$};
 \draw[->](0,-1)--(0,4)node[right]{$f(x)$};
 \node at(-0.2,-0.2){$O$};
 \draw [domain=-1:-0.5,red] plot(\x,-\x*\x+1)node[right]{$f(x)=-(x-k)^2+1$};
 \draw [domain=-1.5:-1,red] plot(\x,-\x*\x-4*\x-3)node[left, yshift=10]{$f(x)=-(x-k)^2-4(x-k)-3$};
 \draw [domain=-2:-1.5,blue] plot(\x,3*\x*\x+6*\x+3)node[left, yshift=40]{$f(x)=3(x+1)^2$};
 \draw [domain=0:0.5,red] plot(\x,-\x*\x+2*\x);
  \draw [domain=-0.5:0,red] plot(\x,-\x*\x-2*\x);
  \draw [domain=1:1.5,red] plot(\x,-\x*\x+4*\x-3);
  \draw [domain=0.5:1,red] plot(\x,-\x*\x+1);
  \draw [domain=2:2.5,red] plot(\x,-\x*\x+6*\x-8);
  \draw [domain=1.5:2,red] plot(\x,-\x*\x+2*\x);
\end{tikzpicture}
\caption{
The function in~\Cref{exp:0}}\label{fig:0}
\end{center}
\end{figure}
\end{example}
\end{remark}
In the following, we also present 
two subgradient upper bounds (B$_{\mbox{\tiny {A-D}}}$) and (B$_{\mbox{\tiny {G}}}$) initially proposed in~\cite{asi2019importance} and~\cite{grimmer2018} respectively.
\begin{definition}{\ }
{\rm
\begin{itemize}
\item[{\rm (B$_{\mbox{\tiny {A-D}}}$).}] {\rm($\dist(x,\setX^*)$-SUB)} $\partial\varphi$ is said to have $\dist(x,\setX^*)$-subgradient upper bound if there are $c_9,\, c_{10}\geq0$ such that for all $x\in\setX$ and $\forall g(x)\in \partial f(x)$ and $\forall h(x)\in \partial r(x)$ it holds that $\|g(x)+h(x)\|\leq c_9\dist(x,\setX^*)+c_{10}$.
\item[{\rm (B$_{\mbox{\tiny {G}}}$).}] {\rm($[\varphi(x)-\varphi^*]$-SUB)}  $\partial\varphi$ is said to have $[\varphi(x)-\varphi^*]$-subgradient upper bound if there are $c_{11},\, c_{12}\geq0$ such that for all $x\in\setX$ and $\forall g(x)\in \partial f(x)$ and $\forall h(x)\in \partial r(x)$ it holds that $\|g(x)+h(x)\|\leq\sqrt{c_{11}[\varphi(x)-\varphi^*]+c_{12}}$.
\end{itemize}
}
\end{definition}

\subsection{Relationship among subgradient upper bounds}
When $f$ is $\rho$-weakly convex and $r$ is $\eta$-weakly convex (thus $\varphi=f+r$ is ($\rho+\eta$)-weakly convex), we have the following relationship 
for the subgradient upper bounds (B$_2$) and (B$_5$):
\begin{proposition}[Relationships for the subgradient upper bounds in the weakly convex case]\label{prop:upperbound2} Suppose $\varphi=f+r$ is ($\rho+\eta$)-weakly convex. Then,
\begin{enumerate}
\item[{\rm (i)}] $\mbox{{\rm(B$_4$)} $\Longrightarrow$ {\rm(B$_3$)} $\Longrightarrow$ {\rm(B$_2$)}.}$
\item[{\rm (ii)}] $\mbox{{\rm(B$_4$)} with $c_6=0 \Longrightarrow$ {\rm(B$_5$)} with $\theta=1$.}$
\end{enumerate}
\end{proposition}
\begin{proof}
{\rm(i)} {\rm (B$_4$) $\Longrightarrow$ (B$_3$):} By statement (iv) of~\Cref{prop:wcvxf_me},~\Cref{cor:me} and the definition of $\varphi_{\lambda}$, for any $x\in\setX$, we have
\begin{center}
$
\begin{aligned}
\|\nabla\varphi_{\lambda}(x)\|^2 =& \frac{1}{\lambda^2}\|x-\prox_{\lambda,\varphi}(x)\|^2 \\
\leq & \frac{2}{\lambda}\left[\varphi(\prox_{\lambda,\varphi}(x))-\varphi^*+\frac{1}{2\lambda}\|x-\prox_{\lambda,\varphi}(x)\|^2\right] \\
= & \frac{2}{\lambda}\left[\varphi_{\lambda}(x)-\varphi^*\right],
\end{aligned}
$
\end{center}
obtaining the desired result.

{\rm (B$_3$) $\Longrightarrow$ (B$_2$):} By the definition of $\varphi_{\lambda}$ and given $x^*\in\setX^*$, for any $x\in\setX$, we have $\varphi_{\lambda}(x)\leq\varphi^*+\frac{1}{2\lambda}\|x-x^*\|^2$. Thus, (B$_3$) yields that
\begin{center}
$
\|g(x)+h(x)\|^2 \leq  c_3[\varphi_{\lambda}(x)-\varphi^*]+c_4
               \leq \frac{c_3}{2\lambda}\|x-x^*\|^2+c_4
               \leq \frac{c_3}{\lambda}\|x\|^2+\frac{c_3\|x^*\|^2}{\lambda}+c_4,
$
\end{center}
implying (B$_2$) to hold.

{\rm (ii)} Trivial.
\end{proof}

For the subgradient upper bounds, a practical issue is: {\it How can one check if a subgradient upper bound holds for one specific function or not?} We observe that, although upper bound (B$_2$) cannot be directly used in convergence analysis for the weakly convex case, in many cases it may be relatively easy to verify. 
Furthermore, we also found that under some condition, upper bound (B$_2$) implies (B$_3$) or (B$_4$); see the next proposition. This relationship can be used to check (B$_3$) or (B$_4$); see also Application 3, Proposition 4.1, and Fact 6.1 of~\cite{zhao2022randomized}. Below let us formally introduce the growth bounds in terms of the distances to the solution sets.

\begin{proposition}\label{prop:upperbound3} Suppose $\varphi=f+r$ is ($\rho+\eta$)-weakly convex. Then, the following assertions hold true:
\begin{itemize}
\item[{\rm (i)}] If $\varphi$ satisfies the global quadratic growth property with parameter $\kappa_{gQG}^{\varphi}>0$ and the set of critical points $\bsetX$ is bounded by constant $\sB_{\bsetX}>0$, then we have {\rm(B$_2$)} $\Longrightarrow$ {\rm(B$_3$)}. In other words, the upper bounding conditions {\rm(B$_2$)} and {\rm(B$_3$)} are equivalent in this case.
\item[{\rm (ii)}] If $\partial\varphi$ satisfies the global metric subregularity property with parameter $\kappa_{gMS}^{\varphi}>0$,
and the set of critical points $\bsetX$ is bounded by constant $\sB_{\bsetX}>0$, then we have {\rm(B$_2$)} $\Longrightarrow$ {\rm(B$_4$)}. In other words, the upper bounding conditions {\rm(B$_2$)} and {\rm(B$_4$)} are also equivalent in this case.
\item[{\rm (iii)}] If there are $\tau$ ($0<\tau<1$) and $R_0>0$ such that $
\|\prox_{\lambda,\varphi}(x)\|\leq \tau \|x\|,\quad\forall x\in\{x\in\setX\mid\|x\|\geq R_0\}$, then {\rm(B$_2$)} $\Longrightarrow$ {\rm(B$_4$)}. In other words, the upper bounding conditions {\rm(B$_2$)} and {\rm(B$_4$)} are again equivalent in this case.
\end{itemize}
\end{proposition}
\begin{proof}
{\rm (i)} By (B$_2$), one easily obtains that
\be\label{eq:as1_as2_1}
\|g(x)+h(x)\|\leq c_1\|x\|+c_2\leq c_1\dist(x,\bsetX)+c_1\|\proj_{\bsetX}(x)\|+c_2, \quad \forall x \in\setX.
\ee
Since the set off critical points $\bsetX$ is bounded by constant $\sB_{\bsetX}>0$,~\eqref{eq:as1_as2_1} yields that
\be\label{eq:as1_as2_2}
\|g(x)+h(x)\|\leq c_1\dist(x,\bsetX)+c_1\sB_{\bsetX}+c_2\leq c_1\dist(x,\bsetX)+c_2',
\ee
with $c_2'=c_1\sB_{\bsetX}+c_2$. Since $\varphi$ satisfies the global quadratic growth property with parameter $\kappa_{gQG}^{\varphi}>0$, we have
\be\label{eq:as1_as2_3_35}
\dist^2(\prox_{\lambda,\varphi}(x),\setX^*)\leq\kappa_{gQG}^{\varphi}[\varphi(\prox_{\lambda,\varphi}(x))-\varphi^*],\qquad\forall x\in\setX.
\ee
By the fact that $\setX^*\subseteq\overline{\setX}$, it follows from~\eqref{eq:as1_as2_3_35} that
\be\label{eq:as1_as2_3_36}
\dist^2(\prox_{\lambda,\varphi}(x),\overline{\setX})\leq\kappa_{gQG}^{\varphi}[\varphi(\prox_{\lambda,\varphi}(x))-\varphi^*],\qquad\forall x\in\setX.
\ee
In combination with
\begin{center}
$
\begin{aligned}
\dist^2(x,\overline{\setX})&\leq\|x-\proj_{\overline{\setX}}(\prox_{\lambda,\varphi}(x))\|^2\\
&\leq2\|x-\prox_{\lambda,\varphi}(x)\|^2+2\dist^2(\prox_{\lambda,\varphi}(x),\overline{\setX}),
\end{aligned}
$
\end{center}
it follows from~\eqref{eq:as1_as2_3_36} that
\begin{center}
$
\begin{aligned}
\dist^2(x,\overline{\setX})&\leq2\kappa_{gQG}^{\varphi}[\varphi(\prox_{\lambda,\varphi}(x))-\varphi^*]+2\|x-\prox_{\lambda,\varphi}(x)\|^2\\
&\leq2\max\{\kappa_{gQG}^{\varphi},2\lambda\}[\varphi(\prox_{\lambda,\varphi}(x))-\varphi^*+\frac{1}{2\lambda}\|x-\prox_{\lambda,\varphi}(x)\|^2]\\
&=2\max\{\kappa_{gQG}^{\varphi},2\lambda\}[\varphi_{\lambda}(x)-\varphi^*].
\end{aligned}
$
\end{center}
Furthermore, it follows from~\eqref{eq:as1_as2_2} that
\be\label{eq:as1_as2_4_35}
\|g(x)+h(x)\|\leq c_1\sqrt{2\max\{\kappa_{gQG}^{\varphi},2\lambda\}[\varphi_{\lambda}(x)-\varphi^*]}+c_2', \quad \forall x \in\setX,
\ee
which yields (B$_3$). Combining with~\Cref{prop:upperbound2}, the desired result follows.

\noindent{\rm(ii)} Since $\partial\varphi$ satisfies the global metric subregularity property with parameter $\kappa_{gMS}^{\varphi}>0$, 
by~\Cref{prop:global_eb}, we have that $\nabla\varphi_{\lambda}$ satisfies the global metric subregularity with constant $\kappa_{gMS}^{\varphi_{\lambda}}>0$:
\be\label{eq:as1_as2_3}
\dist(x,\overline{\setX})\leq\kappa_{gMS}^{\varphi_{\lambda}}\|\nabla\varphi_{\lambda}(x)\|,\qquad\forall x\in\setX.
\ee
Again using~\eqref{eq:as1_as2_2}, we have that
\be\label{eq:as1_as2_4}
\|g(x)+h(x)\|\leq c_1\kappa_{gMS}^{\varphi_{\lambda}}\|\nabla\varphi_{\lambda}(x)\|+c_2', \quad \forall x\in\setX,
\ee
which yields (B$_4$). Combining with~\Cref{prop:upperbound2}, the desired result follows.

\noindent{\rm(iii)} Since $\|x\|\leq\|x-\prox_{\lambda,\varphi}(x)\|+\|\prox_{\lambda,\varphi}(x)\|$, by the conditions we have  $\|x\|\leq\frac{\lambda}{1-\tau}\|\nabla\varphi_{\lambda}(x)\|$. If (B$_2$) holds, we have $\|g(x)+h(x)\|\leq c_1\|x\|+c_2\leq\frac{\lambda c_1}{1-\tau}\|\nabla\varphi_{\lambda}(x)\|+c_2$, $\forall x\in\{x\in\setX \mid \|x\|\geq R_0\}$, which implies that upper bound (B$_4$) holds. Combining with~\Cref{prop:upperbound2}, the desired result thus follows.
\end{proof}
In the following proposition, we shall establish the relationship among conditions (B$_{\mbox{\tiny {A-D}}}$), (B$_{\mbox{\tiny {G}}}$), (B$_2$), (B$_3$) and (B$_6$).
\begin{proposition}\label{prop:bad}
Suppose $\varphi=f+r$ is $(\rho+\eta)$-weakly convex and $\varphi_{\lambda}$ is the Moreau envelope of $\varphi$ with $\lambda<\frac{1}{\rho+\eta}$. Then, we have
\begin{itemize}
\item[{\rm (i)}] {\rm(B$_{\mbox{\tiny G}}$)} $\Longleftrightarrow$ {\rm(B$_3$)}.
\item[{\rm (ii)}] {\rm(B$_{\mbox{\tiny {A-D}}}$)} $\Longrightarrow$ {\rm(B$_2$)}. Conversely, if the optimal solution set $\setX^*$ is bounded, then {\rm (B$_2$)} $\Longrightarrow$ {\rm(B$_{\mbox{\tiny {A-D}}}$)}.
\item[{\rm (iii)}] {\rm(B$_{\mbox{\tiny {A-D}}}$)} $\Longrightarrow$ {\rm(B$_6$)}.
\item[{\rm (iv)}] {\rm(B$_{\mbox{\tiny G}}$)} $\Longrightarrow$ {\rm(B$_{\mbox{\tiny {A-D}}}$)} with $c_9=\sqrt{c_{11}(\rho+\eta+c_{11})}$ and $c_{10}=\sqrt{2c_{12}}$. Conversely, if $\varphi$ satisfies the global quadratic growth property with parameter $\kappa_{gQG}^{\varphi}>0$, i.e., $\dist^2(x,\setX^*)\leq\kappa_{gQG}^{\varphi}[\varphi(x)-\varphi^*]$, $\forall x\in\R^d$, then we have {\rm(B$_{\mbox{\tiny {A-D}}}$)} $\Longrightarrow$ {\rm(B$_{\mbox{\tiny{G}}}$)}.
\end{itemize}
\end{proposition}
\begin{proof}
\noindent{\rm(i)} {\rm(B$_3$)}$\Longrightarrow${\rm(B$_{\mbox{\tiny G}}$)}. This statement directly follows statement (ii) of~\Cref{prop:wcvxf_me}.

{\rm(B$_{\mbox{\tiny G}}$)} $\Longrightarrow$ {\rm(B$_3$)}. 
For any $x\in\setX$, by definition of $\varphi_{\lambda}(x)$, $\prox_{\lambda,\varphi}(x)$, $g(x)\in\partial f(x)$,
$h(x)\in\partial r(x)$, $(\rho+\eta)$-weakly convexity of $\varphi$, and $\lambda<\frac{1}{\rho+\eta}$,
it follows that
\begin{center}
$
\begin{aligned}
 \varphi(x)-\varphi^*=&\varphi_{\lambda}(x)-\varphi^*+\varphi(x)-\left[\varphi(\prox_{\lambda,\varphi}(x))+\frac{1}{2\lambda}\|x-\prox_{\lambda,\varphi}(x)\|^2\right]\\
\leq&\varphi_{\lambda}(x)-\varphi^*+\langle g(x)+h(x),x-\prox_{\lambda,\varphi}(x)\rangle
 +\frac{1}{2}\left(\rho+\eta-\frac{1}{\lambda}\right)\|x-\prox_{\lambda,\varphi}(x)\|^2\\
\leq&\varphi_{\lambda}(x)-\varphi^*+\|g(x)+h(x)\|\cdot\|x-\prox_{\lambda,\varphi}(x)\|\\
\leq&\varphi_{\lambda}(x)-\varphi^*+\frac{1}{2c_{11}}\|g(x)+h(x)\|^2+\frac{c_{11}}{2}\|x-\prox_{\lambda,\varphi}(x)\|^2\\
\leq&(1+c_{11}\lambda)[\varphi_{\lambda}(x)-\varphi^*]+\frac{1}{2c_{11}}\|g(x)+h(x)\|^2.
\end{aligned}
$
\end{center}
Thus, by {\rm(B$_{\mbox{\tiny G}}$)} we have $\|g(x)+h(x)\|^2 \leq c_{11}[\varphi(x)-\varphi^*]+c_{12}
\leq c_{11}(1+c_{11}\lambda)[\varphi_{\lambda}(x)-\varphi^*]+\frac{1}{2}\|g(x)+h(x)\|^2+c_{12}$ leading to $\|g(x)+h(x)\|^2\leq2c_{11}(1+c_{11}\lambda)[\varphi_{\lambda}(x)-\varphi^*]+2c_{12}$, which shows that {\rm(B$_3$)} holds.

\noindent{\rm(ii)} {\rm(B$_{\mbox{\tiny {A-D}}}$)} $\Longrightarrow$ {\rm(B$_2$)}. For any $x\in\setX$ and given $x^*\in\setX^*$, {\rm(B$_{\mbox{\tiny {A-D}}}$)} yields that $
\|g(x)+h(x)\|\leq c_{9}\dist(x,\setX^*)+c_{10}\leq c_{9}\|x-x^*\|+c_{10}\leq c_9\|x\|+c_9\|x^*\|+c_{10}$.

Therefore, {\rm(B$_2$)} holds with $c_1=c_9$ and $c_2=c_9\|x^*\|+c_{10}$.

{\rm (B$_2$)} $\Longrightarrow$ {\rm(B$_{\mbox{\tiny {A-D}}}$)}. Let $x_p=\arg\min_{y\in\setX^*}\|y-x\|$. It follows from {\rm (B$_2$)} that
\begin{center}
$
\begin{aligned}
\|g(x)+h(x)\| \leq & c_1\|x\|+c_2\leq c_1\|x-x_p\|+c_1\|x_p\|+c_2\\
\leq&c_1\dist(x,\setX^*)+c_1\|x_p\|+c_2.
\end{aligned}
$
\end{center}
By the boundedness of $\setX^*$, there exists $c_{10}\geq0$ such that $c_{10}\geq c_1\|x_p\|+c_2$. Then {\rm(B$_{\mbox{\tiny {A-D}}}$)} holds.

\noindent{\rm(iii)} {\rm(B$_{\mbox{\tiny {A-D}}}$)} $\Longrightarrow$ {\rm(B$_6$)}. For any $x\in\setX$ and $\dist(x,\setX^*)<\gamma$ with given $\gamma>0$, by {\rm(B$_{\mbox{\tiny {A-D}}}$)} we have $\|g(x)+h(x)\|\leq c_9\dist(x,\setX^*)+c_{10}\leq c_9\gamma+c_{10}$. Then {\rm(B$_6$)} holds.

\noindent{\rm(iv)} {\rm(B$_{\mbox{\tiny\rm G}}$)} $\Longrightarrow$ {\rm(B$_{\mbox{\tiny {A-D}}}$)}. Let $x_p^*=\arg\min_{y\in\setX^*}\|y-x\|$. By the ($\rho+\eta$)-weakly convex of $\varphi$, we have
\begin{center}
$
\begin{aligned}
\varphi(x)-\varphi^*\leq&\langle g(x)+h(x),x-x_p^*\rangle+\frac{\rho+\eta}{2}\dist^2(x,\setX^*)\\
\leq&\frac{1}{2c_{11}}\|g(x)+h(x)\|^2+\frac{\rho+\eta+c_{11}}{2}\dist^2(x,\setX^*).
\end{aligned}
$
\end{center}
Thus, (B$_{\mbox{\tiny\rm G}}$) yields that
\begin{center}
$
\begin{aligned}
\|g(x)+h(x)\|^2\leq&c_{11}[\varphi(x)-\varphi^*]+c_{12}\\
\leq&\frac{1}{2}\|g(x)+h(x)\|^2+\frac{c_{11}(\rho+\eta+c_{11})}{2}\dist^2(x,\setX^*)+c_{12},
\end{aligned}
$
\end{center}
implying {\rm(B$_{\mbox{\tiny {A-D}}}$)} to hold.

{\rm(B$_{\mbox{\tiny {A-D}}}$)} $\Longrightarrow$ {\rm(B$_{\mbox{\tiny {G}}}$)}. By {\rm(B$_{\mbox{\tiny {A-D}}}$)}, one easily obtains that $\|g(x)+h(x)\|\leq c_9\dist(x,\setX^*)+c_{10}$, $\forall x \in\domain\varphi$. Since $\varphi$ satisfies the global quadratic growth property with parameter $\kappa_{gQG}^{\varphi}>0$, we have {\rm(B$_{\mbox{\tiny {G}}}$)} directly. 
\end{proof}
\begin{remark}
{\rm The subgradient upper bounds {\rm (B$_{\mbox{\tiny {D-D}}}$)}, {\rm(B$_1$)}, {\rm (B$_3$)}, and {\rm (B$_4$)} will be used in the convergence and $O(1/T^{\frac{1}{4}})$ rate analysis in~\Cref{subsect:converge}. The subgradient upper bound {\rm (B$_5$)} will be used in the extended convergence analysis in~\Cref{subsec:extended_rate}. Finally, the subgradient upper bounds {\rm (B$_6$)}, {\rm (B$_{\mbox{\tiny {A-D}}}$)}, and {\rm (B$_{\mbox{\tiny {G}}}$)} will be used in the linear convergence analysis of~\Cref{subsec:linear_converge}.
}
\end{remark}


To help easy referencing, we summarize below the relevant error bound conditions and subgradient upper bounds that have been introduced so far:

\begin{table}[!h]
\begin{center}
\caption{A glance of the error bound conditions }
{\small
\begin{tabular}{ccc}
\toprule
Name of the Condition & Function/Mapping & Bound \\
\toprule
KL with exponent $1/2$ & $\varphi$ & $\varphi(x)-
\bar{\varphi}
\leq\kappa_{\tiny{ KL}}^{\varphi} \dist^2(0,\partial \varphi(x)+\mathcal{N}_{\setX}(x))$ \\
KL with exponent $1/2$ & $\varphi_{\lambda}$ & $\varphi_{\lambda}(x)-
  \bar{\varphi}
  \leq\kappa_{\tiny KL}^{\varphi_{\lambda}} \|\nabla\varphi_{\lambda}(x)\|^2$ \\
Uniform KL with exponent $1/2$ & $\varphi_{\lambda}$ & $\varphi_{\lambda}(x)-\bar{\varphi}\leq\kappa_{\tiny uKL}^{\varphi_{\lambda}} \|\nabla\varphi_{\lambda}(x)\|^2,\, \forall x\in\mathfrak{B}(\bar{\varphi},\epsilon,\nu)$ \\
(gKL) & $\varphi$ &
$\varphi(x)-\varphi^*\leq\kappa_{\tiny gKL}^{\varphi} \dist^2(0,\partial \varphi(x)+\mathcal{N}_{\setX}(x))$ \\
(gKL) & $\varphi_{\lambda}$ &
$\varphi_{\lambda}(x)-\varphi^*\leq\kappa_{\tiny gKL}^{\varphi_{\lambda}}\|\nabla\varphi_{\lambda}(x)\|^2$ \\
(gQG) & $\varphi_{\lambda}$ & $\dist^2(x,\setX^*)\leq\kappa_{\tiny gQG}^{\varphi_{\lambda}}[\varphi_{\lambda}(x)-\varphi^*]$ \\
(gQG) & $\varphi$ & $\dist^2(x,\setX^*)\leq\kappa_{\tiny gQG}^{\varphi}[\varphi(x)-\varphi^*]$ \\
(gMS) & $\partial \varphi$ & $\dist(x,\overline{\setX})\leq\kappa_{\tiny gMS}^{\varphi}\dist(0,\partial\varphi(x)+\mathcal{N}_{\setX}(x))$ \\
(gMS) & $\nabla \varphi_{\lambda}$ & $\dist(x,\overline{\setX})\leq\kappa_{\tiny gMS}^{\varphi_{\lambda}}\|\nabla\varphi_{\lambda}(x)\|$ \\
\bottomrule
\end{tabular}
}
\end{center}
\end{table}
\begin{table}[!h]
\begin{center}
\caption{A glance of the subgradient upper bound conditions }
{\small
\begin{tabular}{lc}
\toprule
\qquad\qquad\; Name & Bound \\
\toprule
{\rm (B$_{\mbox{\tiny {D-D}}}$): $f$-Lipschitz SUB} &
$\|g(x)\|\leq L_f$\\
{\rm (B$_1$): $r$-Lipschitz SUB} &  $\|h(x)\|\leq L_r$\\
{\rm (B$_2$): linear bounded SUB} &  $\|g(x)+h(x)\|\leq c_1\|x\|+c_2$\\
{\rm (B$_3$): $[\varphi_{\lambda}(x)-\varphi^*]$-SUB} &
$\|g(x)+h(x)\|\leq\sqrt{c_3[\varphi_{\lambda}(x)-\varphi^*]+c_4}$\\
{\rm (B$_4$): $\|\nabla\varphi_{\lambda}(x)\|$-SUB} &
$\|g(x)+h(x)\|\leq\sqrt{c_5\|\nabla\varphi_{\lambda}(x)\|^2+c_{6}}$\\
{\rm (B$_5$): $\|\nabla\varphi_{\lambda}(x)\|$-exponent SUB}&
$\|g(x)\|\leq c_7\|\nabla\varphi_{\lambda}(x)\|^{\theta}, \theta\in[1,2)$\\
{\rm (B$_6$): $\mathcal{T}_{\gamma}$-SUB} &
$\|g(x)+h(x)\|\leq c_8$\\
{\rm (B$_{\mbox{\tiny {A-D}}}$): $\dist(x,\setX^*)$-SUB}& $\|g(x)+h(x)\|\leq c_9\dist(x,\setX^*)+c_{10}$\\
{\rm (B$_{\mbox{\tiny {G}}}$): $[\varphi(x)-\varphi^*]$-SUB}&
$\|g(x)+h(x)\|\leq\sqrt{c_{11}[\varphi(x)-\varphi^*]+c_{12}}$\\
\bottomrule
\end{tabular}
}
\end{center}
\end{table}

The relationships of the subgradient upper bound conditions discussed so far
can be
visualized in
\Cref{fig:2}.
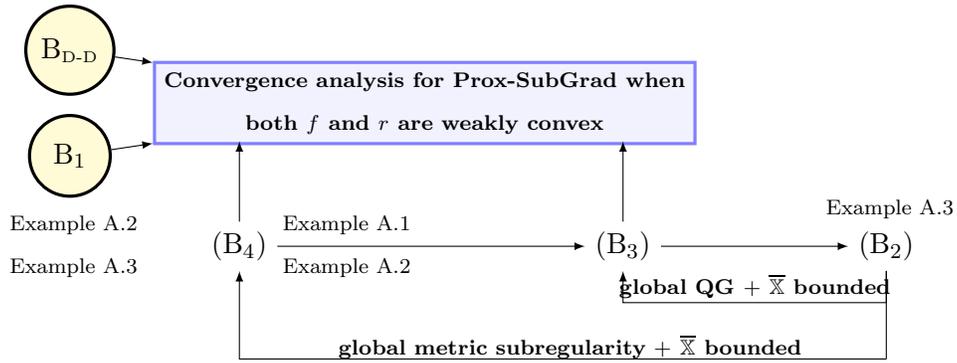
\begin{figure}[h]
\centering
\begin{center}
\begin{tikzpicture}
[
>=latex,
node distance=5mm,
 ract/.style={draw=blue!50, fill=blue!5,rectangle,minimum size=6mm, very thick, font=\itshape, align=center},
 racc/.style={rectangle, align=center},
 ractm/.style={draw=red!100, fill=red!5,rectangle,minimum size=6mm, very thick, font=\itshape, align=center},
 cirl/.style={draw, fill=yellow!20,circle,   minimum size=6mm, very thick, font=\itshape, align=center},
 raco/.style={draw=green!500, fill=green!5,rectangle,rounded corners=2mm,  minimum size=6mm, very thick, font=\itshape, align=center},
 hv path/.style={to path={-| (\tikztotarget)}},
 vh path/.style={to path={|- (\tikztotarget)}},
 skip loop/.style={to path={-- ++(0,#1) -| (\tikztotarget)}},
 vskip loop/.style={to path={-- ++(#1,0) |- (\tikztotarget)}}]

\node (b3) [font=\Large] at ( 1,0)     {\normalsize(B$_2$)};

\node (b6) [font=\Large] at ( -3.5,0)    {\normalsize(B$_3$)$=$(B$_{\mbox{\rm\tiny {G}}}$)};
\node (b66) [font=\Large] at ( -1.1,0)    {\normalsize(B$_{\mbox{\rm\tiny {A-D}}}$)};
\node (b666) [font=\Large] at ( -1.1,1.2)    {\normalsize(B$_6$)};
\node (b44) [racc, below of = b6, xshift = 63, yshift=-17]{\scriptsize\bf{global QG $+$ $\overline{\setX}$ bounded}};
\node (b444) [racc, below of = b6, xshift = 38, yshift=5]{\scriptsize\bf{global QG}};
\node (b444) [racc, below of = b6, xshift = 102, yshift=6]{\scriptsize\bf{$\setX^*$ bounded}};
\node (b45) [racc, below of = b6, xshift =35, yshift=-38]{\scriptsize  \bf{global metric subregularity $+$ $\overline{\setX}$ bounded}};
\node (b7) [font=\Large] at ( -5.6,0)    {\normalsize(B$_4$)};
\node (b8) [font=\Large] at (-8.2,0)    {\normalsize(B$_5$) with $\theta=1$};
\node (b85) [racc, above of = b8, xshift =50, yshift=-9]{\scriptsize  \bf{$c_6=0$}};
\path (b7) edge[->] (b6);
\path (b66) edge[->] (b666);
\path (b7) edge[->] (b8);
\path (b6) edge[->] (b66);
\path (-1.75,-0.1) edge[->] (-2.6,-0.1);
\path (b66) edge[->] (b3);
\path (0.5,-0.1) edge[->] (-0.5,-0.1);
\path (b3) edge[-] (1,-1.25);
\path (b3) edge[-] (1,-2);
\path (1,-1.25) edge[->, hv path] (b6);
\path (1,-2) edge[->, hv path] (b7);
\end{tikzpicture}
\caption{Relationships among the subgradient upper bounds} 
\label{fig:2}
\end{center}
\end{figure}
\section{Convergence analysis} 
\label{sec:wcvx_wcvx}
Based on the subgradient upper bound conditions (B$_{\mbox{\tiny {D-D}}}$), (B$_1$) and (B$_3$) and (B$_4$), we shall now develop a unified convergence analysis whenever $\varphi$ is weakly convex.
\subsection{Uniform recursion of Prox-SubGrad}
At the beginning of this subsection, we establish the basic recursion of Prox-SubGrad for the weakly convex case. The technical proofs of this subsection are relegated to Appendix~\ref{appendix-c}.
\begin{lemma}[Basic recursion]\label{lemma:basic_recur_wcvx}
Suppose $f$ is $\rho$-weakly convex, $r$ is $\eta$-weakly convex, and the positive number $\lambda<\frac{1}{\rho+2\eta}$. Let $\{x^k\}_{k\in \N}$ be the sequence generated by Prox-SubGrad for solving problem~\eqref{eq:opt problem_prox}. Then, for any $x\in\setX$, the following inequalities hold:
\begin{itemize}
\item[{\rm(i)}] If $\alpha_k<\frac{1}{6\eta}$, it holds that
\begin{center}
$
\begin{aligned}
\|x-x^{k+1}\|^2\leq&\left(1-\frac{(1-\lambda(\rho+2\eta))\alpha_k}{\lambda}\right)\|x-x^k\|^2-2\alpha_k\left[\varphi(x^k)-\varphi(x)-\frac{1}{2\lambda}\|x-x^k\|^2\right] \\
 &        +2\alpha_k^2\|g(x^k)+h(x^k)\|^2.
\end{aligned}
$
\end{center}
\item[{\rm(ii)}] Similarly, if $\alpha_k<\frac{1}{2(\rho+2\eta+\frac{1}{\lambda})}$, it holds that
\begin{center}
$
\begin{aligned}
\|x-x^{k+1}\|^2 \leq&\left(1-\frac{(1-\lambda(\rho+2\eta))\alpha_k}{\lambda}\right)\|x-x^k\|^2-2\alpha_k\left[\varphi(x^{k+1})-\varphi(x)-\frac{1}{2\lambda}\|x-x^k\|^2\right] \\
 &        -\frac{\alpha_k}{\lambda}\|x^k-x^{k+1}\|^2+2\alpha_k^2\|g(x^k)-g(x^{k+1})\|^2.
\end{aligned}
$
\end{center}
\item[\rm(iii)] If $\alpha_k\leq\frac{1}{2\eta}$, we have $\|x^k-x^{k+1}\|\leq2\alpha_k\|g(x^k)+h(x^k)\|$.
\end{itemize}
\end{lemma}
\begin{lemma}[Recursion with $x=\hat{x}^k=\prox_{\lambda,\varphi}(x^k)$ and $x=x^*$]\label{lemma:basic_recur_wcvx_hx}
Suppose $f$ is $\rho$-weakly convex, $r$ is $\eta$-weakly convex and the positive number $\lambda<\frac{1}{\rho+2\eta}$. Let $\{x^k\}_{k\in \N}$ be the sequence generated by Prox-SubGrad for solving problem~\eqref{eq:opt problem_prox}.
\begin{itemize}
\item[{\rm(i)}]  If $\alpha_k<\frac{1}{6\eta}$, it holds that
\begin{center}
$\|\hat{x}^k-x^{k+1}\|^2
\leq\left(1-\frac{(1-\lambda(\rho+2\eta))\alpha_k}{\lambda}\right)\|\hat{x}^k-x^k\|^2+\underbrace{2\alpha_k^2\|g(x^k)+h(x^k)\|^2}_{T_1}$.
\end{center}
\item[{\rm(ii)}] Similarly, if $\alpha<\frac{1}{2(\rho+2\eta+\frac{1}{\lambda})}$, it holds that
\begin{center}
$\|\hat{x}^k-x^{k+1}\|^2 \leq\left(1-\frac{(1-\lambda(\rho+2\eta))\alpha_k}{\lambda}\right)\|\hat{x}^k-x^k\|^2+\underbrace{2\alpha_k^2\|g(x^k)-g(x^{k+1})\|^2}_{T_2}$.
\end{center}
\item[{\rm(iii)}]  If $\alpha_k<\frac{1}{6\eta}$, it holds that
\begin{equation}\label{eq:recur_linear}
\dist^2(x^{k+1},\setX^*)
\leq\left(1+(\rho+2\eta)\alpha_k\right)\dist^2(x^k,\setX^*)-2\alpha_k\left[\varphi(x^k)-\varphi^*\right]+\underbrace{2\alpha_k^2\|g(x^k)+h(x^k)\|^2}_{T_1}.
\end{equation}
\end{itemize}
\end{lemma}
The next lemma allows us to establish a unified Moreau envelope recursive relationship regarding $\varphi_{\lambda}(x^k)-\varphi^*$ for all subgradient upper bounding conditions (B$_{\mbox{\tiny {D-D}}}$), (B$_1$) and (B$_3$) and (B$_4$). Note that~\Cref{prop:upperbound3} shows that we need only to analyze upper bounding condtions (B$_{\mbox{\tiny {D-D}}}$), (B$_1$) and (B$_3$) for weakly convex case.

\begin{lemma}\label{lemma:profound_recur_wcvx}{\bf(Moreau envelope recursive relations 
under various subgradient upper bounding conditions)}
Suppose $f$ is $\rho$-weakly convex and $r$ is $\eta$-weakly convex. Let $\{x^k\}_{k\in \N}$ be the sequence generated by Prox-SubGrad for solving problem~\eqref{eq:opt problem_prox}, and the sequence $\{\beta_k\}$ is a $\sigma$-sequence, meaning $\beta_k>0$, $\sum_{k\in \N} \beta_k=\infty$, and $\sum_{k\in \N}\beta_k^2<\infty$. Then, for the subgradient upper bound scenarios {\rm(B$_{\mbox{\rm\tiny {D-D}}}$)}, {\rm(B$_1$)} and {\rm(B$_3$)}, we uniformly have the recursion with $\gamma_1$, $\gamma_2\geq0$ as
\be\label{eq:uniform_recur_wcvx}
\begin{aligned}
[\varphi_{\lambda}(x^{k+1})-\varphi^*]-(1+\gamma_1\beta_k^2)[\varphi_{\lambda}(x^k)-\varphi^*]\leq-\frac{(1-\lambda(\rho+2\eta))\alpha_k}{2}\|\nabla\varphi_{\lambda}(x^{k})\|^2+\gamma_2\beta_k^2.
\end{aligned}
\ee
\end{lemma}
\begin{proof} For all subgradient upper bound scenario {\rm (B$_{\mbox{\tiny\rm D-D}})$}, {\rm (B$_{1}$)}, and {\rm (B$_{3}$)}, taking stepsize $\alpha_k$ as in the table below, the terms $T_1$ and $T_2$ as in~\Cref{lemma:basic_recur_wcvx_hx} can be estimated accordingly:

\smallskip

\begin{tabular}{cccc}
\toprule
SubGrad UB & Stepsize  &  $T_1$ & $T_2$\\
\midrule 
{\rm (B$_{\mbox{\tiny\rm D-D}}$)} & $\alpha_k=\beta_k<\frac{1}{2(\frac{1}{\lambda}+2\eta+\nu)}$ &   --     &   $T_2\leq 8L_f^2\alpha_k^2$\\
{\rm (B$_{1}$)} & $\alpha_k=\frac{\beta_k}{\|g(x^k)\|+L_r}$ & $T_1\leq2\beta_k^2$ & --\\
{\rm (B$_{3}$)} & $\alpha_k=\beta_k<\frac{1}{6\eta}$ & $T_1\leq 2\beta_k^2c_3[\varphi_{\lambda}(x^k)-\varphi^*]+2\beta_k^2c_4$& --\\
\bottomrule 
\end{tabular}

\smallskip

Finally, the desired result follows from the above table, statements (i) and (ii) of~\Cref{lemma:basic_recur_wcvx_hx}, and~\Cref{lemma:phi_lambda_2}.
\end{proof}
\subsection{The convergence analysis of Prox-SubGrad with $\sigma$-sequence stepsize for nonsmooth weakly convex optimization}\label{subsect:converge}
In this subsection, we will derive an $O(1/\sqrt{T})$ rate of convergence in terms $\|\varphi_{\lambda}(x^k)\|^2$ based on the subgradient upper bound conditions (B$_{\mbox{\tiny {D-D}}}$), (B$_3$) and (B$_4$). Furthermore, if the uniform KL property with exponent $1/2$ holds, then we obtain an $O(1/T)$ rate of iteration complexity bound.
\begin{theorem} 
[Convergence analysis with \{{\rm(B$_{\mbox{\rm\tiny {D-D}}}$)}, {\rm(B$_3$)}, {\rm(B$_4$)}\}]\label{theo:convergence_wcvx}
Suppose $f$ is $\rho$-weakly convex and $r$ is $\eta$-weakly convex with $\rho$, $\eta\geq0$. Let $\{x^k\}_{k\in \N}$ be the sequence generated by Prox-SubGrad for solving problem~\eqref{eq:opt problem_prox}. If one of subgradient upper bound conditions \{{\rm(B$_{\mbox{\rm\tiny {D-D}}}$)}, {\rm(B$_3$)}, {\rm(B$_4$)}\} holds, and  the sequence $\{\beta_k\}$ and the stepsizes $\{\alpha_k\}$ are selected as shown in~\Cref{lemma:profound_recur_wcvx}, and the parameter $\lambda<\frac{1}{\rho+2\eta}$, then the following assertions hold true:
\begin{enumerate}[label=\textup{\textrm{(\alph*)}},topsep=0pt,itemsep=0ex,partopsep=0ex]
\item[{\rm (i)}] $\min\limits_{0\leq k\leq T}\|\nabla\varphi_{\lambda}(x^k)\|^2\leq\left(\frac{2}{1-\lambda(\rho+2\eta)}\right)\frac{[\varphi_{\lambda}(x^{0})-\varphi^*]+\gamma_3\sum_{k=0}^{T}\beta_k^2}{\sum_{k=0}^{T}\alpha_k}$ with $\gamma_3\geq\gamma_1[\varphi_{\lambda}(x^k)-\varphi^*]+\gamma_2$. If constant stepsize $\beta_k=\Delta_1/\sqrt{T+1}$ with 
$\Delta_1 > 0$ ($k=0,1,...,T$) is used, then $\min\limits_{0\leq k\leq T}\|\nabla\varphi_{\lambda}(x^k)\|^2\leq R_1/\sqrt{T+1}$ for some $R_1>0$.
\item[{\rm (ii)}] For the subgradient upper bounds {\rm(B$_3$)} and {\rm(B$_4$)}, there exists positive number

$M_1=\max\left\{\sqrt{\frac{c_3\gamma_3}{\gamma_1}+c_4},\sqrt{\frac{4c_{5}\gamma_3^2}{\gamma_1^2\lambda^2}+c_6}\right\}>0$ such that $\|g(x^k)+h(x^k)\|\leq M_1$, and $\|x^k-x^{k+1}\|\leq 2M_1\alpha_k$, and $\lim_{k\rightarrow\infty}\|\nabla\varphi_{\lambda}(x^k)\|=0$.
\end{enumerate}
\end{theorem}
\begin{proof}
\noindent{\rm(i)} By the recursive relation~\eqref{eq:uniform_recur_wcvx} and statement (i) of~\Cref{lemma:Polyak_extend}, we obtain $\varphi_{\lambda}(x^k)-\varphi^*\rightarrow\overline{\mathfrak{a}}\geq0$ which from definition of $\varphi_{\lambda}(x^k)$ shows the boundness of $\{\varphi(\hat{x}^k)\}$ and $\{\|x^k-\hat{x}^k\|\}$. Therefore, there is positive number $\gamma_3\geq\gamma_1[\varphi_{\lambda}(x^k)-\varphi^*]+\gamma_2$ such that
\[
[\varphi_{\lambda}(x^{k+1})-\varphi^*]-[\varphi_{\lambda}(x^k)-\varphi^*]\leq-\frac{(1-\lambda(\rho+2\eta))\alpha_k}{2}\|\nabla\varphi_{\lambda}(x^k)\|^2+\gamma_3\beta_k^2.
\]
Choosing $\alpha_k=\beta_k$ and $\beta_k=\Delta_1/\sqrt{T+1}$,
the result follows.


\noindent{\rm(ii)} By the definition of $\gamma_3$, for $k=0, 1,...$, we have $\varphi_{\lambda}(x^k)-\varphi^*\leq\frac{\gamma_3}{\gamma_1}$ and $\|\nabla\varphi_{\lambda}(x^k)\|\leq\frac{2\gamma_3}{\gamma_1\lambda}$, the subgradient upper bounds {\rm(B$_3$)} and {\rm(B$_4$)} imply that there exists positive number $M_1=\max\left\{\sqrt{\frac{c_3\gamma_3}{\gamma_1}+c_4},\sqrt{\frac{4c_{5}\gamma_3^2}{\gamma_1^2\lambda^2}+c_6}\right\}>0$ such that $\|g(x^k)+h(x^k)\|\leq M_1$. Then, (iii) of~\Cref{lemma:basic_recur_wcvx} yields $\|x^k-x^{k+1}\|\leq 2M_1\alpha_k$.


For the stepsize choice 
in~\Cref{lemma:profound_recur_wcvx},
the sequence $\{\alpha_k\}$ is also a $\sigma$-sequence: $\alpha_k>0$, $\sum_{k\in \N} \alpha_k=\infty$, and $\sum_{k\in \N}\alpha_k^2<\infty$. We can easily derive $\sum_{k\in\setN}\alpha_k\|\nabla\varphi_{\lambda}(x^k)\|^2<\infty$ by~\Cref{lemma:Polyak_extend}. By combining $\|x^k-x^{k+1}\|\leq 2M_1\alpha_k$ with~\Cref{lemma:rcs}, we have $\|\nabla\varphi_{\lambda}(x^k)\|\rightarrow 0$.
\end{proof}


We note here that the $O(1/\sqrt{T})$ rate in terms of $\|\nabla\varphi_{\lambda}(x^k)\|^2$, 
as claimed in~\Cref{theo:convergence_wcvx} is in fact an $O(T^{-1/4})$ convergence rate in terms $\|\nabla\varphi_{\lambda}(x^k)\|$ of~\cite{davis2019stochastic}. The following corollary exhibits the behavior that initially the convergence can be faster,
when $\min\limits_{0\leq k\leq T}\|\nabla\varphi_{\lambda}(x^k)\|^2$ is not too small.

\begin{corollary}\label{cor:convergence_wcvx}
Suppose the assumptions of~\Cref{theo:convergence_wcvx} hold. By {\rm (i)} of the above theorem, for any given positive non-decreasing function $G:\R_+\rightarrow\R_+$ satisfying $G(t)=t$ when $0\leq t<\delta$ and $G(t)>t$ when $t\geq\delta$ with $\delta>0$, we have $\min\limits_{0\leq k\leq T}\|\nabla\varphi_{\lambda}(x^k)\|^2\leq R_1/\sqrt{T+1}\leq G(R_1/\sqrt{T+1})$.
\end{corollary}
For example, if $G(t)=\left\{\begin{array}{cc}\frac{1}{\delta}t^2,&t\geq\delta;\\ t,&0\leq t\leq\delta,
\end{array}\right.$ with $\delta\in(0,1)$, then the term $\min\limits_{0\leq k\leq T}\|\nabla\varphi_{\lambda}(x^k)\|^2$ exhibits an extended $\frac{1}{\delta}R_1^2/\varepsilon$ complexity in the range $[\delta,\infty)$.
Let $\varphi_{\lambda}(\bar{x})=\bar{\varphi}$ be one local minima of $\varphi$. In the following theorem, we consider the case where $\varphi_{\lambda}(x^k)$ converges to the local minima $\bar{\varphi}$. In that case, the uniform recursion~\eqref{eq:uniform_recur_wcvx} also leads to an $O(1/k)$ rate 
in terms of $\varphi_{\lambda}(x)-\bar{\varphi}$ under the uniform KL condition.
We refer the term $\varphi_{\lambda}(x)-\bar{\varphi}$ to be the local suboptimality gap of $\varphi_{\lambda}(x)$.
\begin{theorem}[Convergence rate under additional uniform KL condition]\label{cor:qg_wcvx} Suppose $f$ is $\rho$-weakly convex and $r$ is $\eta$-weakly convex with $\rho$, $\eta\geq0$. Suppose further the sequence $\{\varphi_{\lambda}(x^k)\}$ which is generated by Prox-SubGrad converges to $\bar{\varphi}$, and $\varphi_{\lambda}$ satisfies uniform KL property at $x\in\Omega$, i.e., there is $\kappa_{uKL}^{\varphi_{\lambda}}>0$ such that $\|\nabla\varphi_{\lambda}(x)\|^2\geq\kappa_{uKL}^{\varphi_{\lambda}}[\varphi_{\lambda}(x)-\bar{\varphi}]$, $\forall x\in\mathfrak{B}(\bar{\varphi},\epsilon,\nu)=\{x \mid \dist(x,\Omega)\leq\epsilon,\quad\bar{\varphi}<\varphi_{\lambda}(x)<\bar{\varphi}+\nu\}$ when $\Omega=\{x \mid \varphi_{\lambda}(x)=\bar{\varphi}\}$. Then, under any of the subgradient upper bound conditions \{{\rm(B$_{\mbox{\tiny {D-D}}}$)}, {\rm(B$_3$)}, {\rm(B$_4$)}\}, and if the diminishing stepsizes and $\lambda$ are selected as shown in~\Cref{lemma:profound_recur_wcvx} and $\beta_k=\frac{\Delta_2}{k}$ with sufficient large positive number $\Delta_2$, then the following convergence rate holds 
\[
\varphi_{\lambda}(x^k)-\bar{\varphi}=O\left(\frac{1}{k}\right) .
\]
\end{theorem}
\begin{proof} Since $\varphi_{\lambda}(x^k)\rightarrow\bar{\varphi}$ we have $\gamma_3\geq\gamma_1[\varphi_{\lambda}(x^k)-\varphi^*]+\gamma_2$. Additionally, since $\bar{\varphi}$ is a local minimum, there is $k_0>0$ such that $\{x^k\}\subset\mathfrak{B}(\bar{\varphi},\epsilon,\eta)$ for some $\epsilon$, $\eta>0$ and $k\geq k_0$. Since $\varphi_{\lambda}$ satisfies the uniform KL property, the uniform recursion~\eqref{eq:uniform_recur_wcvx} yields
\[
[\varphi_{\lambda}(x^{k+1})-\bar{\varphi}]-\left(1-\frac{(1-\lambda(\rho+2\eta))\kappa_{uKL}^{\varphi_{\lambda}}\alpha_k}{2}\right)[\varphi_{\lambda}(x^k)-\bar{\varphi}]\leq\gamma_3\beta_k^2.
\]
Note that subgradient upper bound (B$_4$) implies (B$_3$), and for the subgradient upper bound scenarios (B$_{\mbox{\tiny {D-D}}}$), (B$_3$) and (B$_4$), letting $\beta_k=\frac{\Delta_2}{k}$ with $\Delta_2>\frac{4}{(1-\lambda(\rho+2\eta))\kappa_{uKL}^{\varphi_{\lambda}}}$, we obtain $\varphi_{\lambda}(x^{k+1})-\bar{\varphi}]-\left(1-\frac{2}{k}\right)[\varphi_{\lambda}(x^k)-\bar{\varphi}]\leq\frac{\gamma_3\Delta_2^2}{k^2}$. Then, by~\Cref{lemma:Polyak1} we obtain the desired result.
\end{proof}
\begin{remark}
{\rm We noted that the condition $\varphi_{\lambda}(x^k)\rightarrow\bar{\varphi}$ of~\Cref{cor:qg_wcvx} is easy to prove, when $\varphi$ is level-bounded on $\setX$.}
\end{remark}

\begin{remark}\label{rem:B1} {\rm With the level-boundedness in hands, we can easily get an $O(1/\sqrt{T})$ rate as in~\Cref{theo:convergence_wcvx} and an $O(1/k)$ rate as in~\Cref{cor:qg_wcvx} under scenario {\rm (B$_1$)}. }
\end{remark}

\subsection{The extended convergence analysis of Proj-SubGrad with normalized stepsize for nonsmooth weakly convex optimization}\label{subsec:extended_rate}
In this subsection, we consider  problem~\eqref{eq:opt problem_prox} with $r=0$:
\begin{equation}\label{eq:opt problem_proj}
\min_{x\in \setX} \ 
f(x),
\end{equation}
where $f$ is $\rho$-weakly convex. In other words,
$\varphi=f$ in this subsection,
and the algorithm Prox-SubGrad with $\alpha_k=\frac{\beta_k}{\|g(x^k)\|}$ becomes the following Proj-SubGrad scheme:
\[
x^{k+1}:=\arg\min_{x\in\setX} \, \langle g(x^k),x\rangle+\frac{1}{2\alpha_k}\|x-x^k\|^2\quad\mbox{or}\quad x^{k+1}:=\proj_{\setX}\left(x^k-\frac{\beta_k}{\|g(x^k)\|}g(x^k)\right).
\]
We note that, for problem~\eqref{eq:opt problem_proj} and Proj-SubGrad with normalized stepsize $\alpha_k=\frac{\beta_k}{\|g(x^k)\|}$, when $f$ is convex,~\cite{Li2023subgrad} establishs the $O(1/\sqrt{T})$ rate without any other assumptions. Moreover,~\cite{grimmer2019} provides an extended rate under an additional function upper bound condition.
When $f$ is weakly convex,~\cite{Li2023subgrad} shows the $O(1/\sqrt{T})$ rate in terms $\|\nabla\varphi_{\lambda}(x^k)\|^2$ by using additional level boundedness of $f$.
However, it turns out that it is possible to derive a faster rate of convergence, if some additional subgradient upper bounding conditions hold. In the following theorem, we introduce such a result under the subgradient upper bound condition $\|\nabla\varphi_{\lambda}(x)\|$-exponent SUB {\rm(B$_5$)}.

\begin{theorem} 
[Extended convergence analysis for Proj-SubGrad with normalized stepsize]\label{theo:convergence_wcvx_norm}
Suppose $f$ is $\rho$-weakly convex ($\rho\geq0$) and $r=0$. Let $\{x^k\}_{k\in \N}$ be the sequence generated by Proj-SubGrad for solving problem~\eqref{eq:opt problem_proj}. If the subgradient upper bound condition {\rm(B$_5$)} holds, and $\beta_k=\Delta_3/\sqrt{T+1}$ with $\Delta_3>0$, $T>k$ and the parameter $\lambda<\frac{1}{\rho}$, then:
\[
\min\limits_{0\leq k\leq T}\|\nabla\varphi_{\lambda}(x^k)\|\leq\left(\frac{R_2}{\sqrt{T+1}}\right)^{\frac{1}{2-\theta}},\quad\mbox{with}\quad\theta\in[1,2)
\]
where $R_2>0$ is a constant.
\end{theorem}
\begin{proof}
The statement (i) of~\Cref{lemma:basic_recur_wcvx_hx} with $r=0$ and $\alpha_k=\frac{\beta_k}{\|g(x^k)\|}$ follows that
\be\label{eq:recur_hx_r0_1}
\|\hat{x}^k-x^{k+1}\|^2
\leq\left(1-\frac{(1-\lambda\rho)\frac{\beta_k}{\|g(x^k)\|}}{\lambda}\right)\|\hat{x}^k-x^k\|^2+2\beta_k^2.
\ee
By the combination of~\Cref{lemma:phi_lambda_2} and~\eqref{eq:recur_hx_r0_1} and statement (iv) of~\Cref{prop:wcvxf_me}, we have that
\be\label{eq:recur_hx_r0_3}
\begin{aligned}
\varphi_{\lambda}(x^{k+1})-\varphi_{\lambda}(x^k)\leq&-\frac{(1-\lambda\rho)\frac{\beta_k}{\|g(x^k)\|}}{2}\|\nabla\varphi_{\lambda}(x^k)\|^2+\frac{\beta_k^2}{\lambda^2}\\
\overset{\mbox{\small \rm(B$_5$)}}{\leq}
&-\frac{(1-\lambda\rho)\beta_k}{2c_7}\|\nabla\varphi_{\lambda}(x^k)\|^{2-\theta}+\frac{\beta_k^2}{\lambda^2}.
\end{aligned}
\ee
Unrolling this recursion~\eqref{eq:recur_hx_r0_3} for $k=0, 1, ..., T$, and rearranging the terms provide
\[
\sum_{k=0}^{T}\frac{(1-\lambda\rho)\beta_k}{2c_7}\|\nabla\varphi_{\lambda}(x^k)\|^{2-\theta}\leq\varphi_{\lambda}(x^0)-\varphi^*+\sum_{k=0}^T\frac{\beta_k^2}{\lambda^2}.
\]
Since $\theta\in[1,2)$, function $t^{2-\theta}$ is increasing when $t>0$. Then we have
\[
\left[\min_{0\leq k\leq T}\|\nabla\varphi_{\lambda}(x^k)\|\right]^{2-\theta}\leq\frac{\varphi_{\lambda}(x^0)-\varphi^*+\sum_{k=0}^T\frac{\beta_k^2}{\lambda^2}}{\sum_{k=0}^{T}\frac{(1-\lambda\rho)\beta_k}{2c_7}}.
\]
Let
$\beta_k=\frac{\Delta_3}{\sqrt{T+1}}$. We have $\min\limits_{0\leq k\leq T}\|\nabla\varphi_{\lambda}(x^k)\|\leq\left(\frac{2c_7[\varphi_{\lambda}(x^0)-\varphi^*+\frac{\Delta_3^2}{\lambda^2}]}{(1-\lambda\rho)\Delta_3\sqrt{T+1}}\right)^{\frac{1}{2-\theta}}$.
\end{proof}
We note that since $\theta\in[1,2)$,
~\Cref{theo:convergence_wcvx_norm} gives rise to a faster rate of convergence for Proj-SubGrad method.
\subsection{
Linear convergence}
\label{subsec:linear_converge}
Under some conditions,
Proj-SubGrad may even exhibit the behavior of linear convergence. This subsection is devoted to such an analysis.
Let us begin a new error bound condition as follows:

\begin{definition}{\ }
{\rm
 $\varphi$ is said to satisfy the $\mathcal{T}_{\gamma}$-sharpness condition
 if there are $\kappa_{\tiny\rm s}^{\varphi}>0$ and $\gamma>0$ such that $\dist(x,\setX^*)\leq\kappa_{\tiny\rm s}^{\varphi}[\varphi(x)-\varphi^*], \,\forall x\in\mathcal{T}_{\gamma} :=\{x\in\setX\mid \dist(x,\setX^*)<\gamma\}$.
}
\end{definition}

\begin{theorem}[Linear Convergence under additional condition]\label{cor:qg_wcvx_2}
Suppose $f$ is $\rho$-weakly convex and $r$ is $\eta$-weakly convex with $\rho$, $\eta\geq0$. Now, consider either of the following two conditions {\rm (i)} or {\rm (ii)} holds. Specifically,
\begin{itemize}
\item[{\rm(i)}] Suppose further $\varphi$ satisfies the global quadratic growth  property with $\kappa_{gQG}^{\varphi}>0$, 
and suppose that one of the following three conditions holds:

\noindent{\rm (1)} $\kappa_{gQG}^{\varphi}<\frac{2}{\rho+2\eta}$, and the subgradient upper bound condition {\rm(B$_{\mbox{\tiny {A-D}}}$)} holds with $c_{10}=0$, and we select the constant stepsize $\alpha_k=\alpha\in\left(0,\min\left\{\frac{2-\kappa_{gQG}^{\varphi}(\rho+2\eta)}{2c_9^2\kappa_{gQG}^{\varphi}},\frac{1}{6\eta}\right\}\right)$.

\noindent{\rm (2)} $\kappa_{gQG}^{\varphi}<\frac{2}{\rho+2\eta}$, and the subgradient upper bound condition {\rm(B$_{\mbox{\tiny {G}}}$)} holds with $c_{12}=0$, and we select
 the constant stepsize $\alpha_k=\alpha\in\left(0,\min\left\{\frac{2-(\rho+2\eta)}{2c_{11}(\rho+\eta+c_{11}){\kappa_{gQG}^{\varphi}}},\frac{1}{6\eta}\right\}\right)$.

\noindent{\rm (3)} If $\kappa_{gQG}^{\varphi}<\frac{1}{\rho}$, and $\eta=0$, and if we select the modified Polyak stepsize $\alpha_k=\frac{\varphi(x^k)-\varphi^*}{2\|g(x^k)+h(x^k)\|^2}$, and if $\varphi$ satisfies the  $\mathcal{T}_{\gamma}$-sharpness with constant $\kappa_{\tiny\rm s}^{\varphi}>0$
and {\rm(B$_6$)} holds on the tube $\mathcal{T}_{\gamma_0}:=
\{x\in\setX\mid \dist(x,\setX^*)<\dist(x^0,\setX^*)+1\}$.
\item[{\rm(ii)}] We have $\eta=0$, and $\varphi$ is $\mathcal{T}_{\gamma}$-sharpness with constant $\kappa_{\tiny\rm s}^{\varphi}>0$ on $\mathcal{T}_{\gamma}$ with $0<\gamma<\frac{\kappa_{\rm\tiny s}}{\rho}$. Moreover, {\rm(B$_6$)} holds on $\mathcal{T}_{\gamma}$ and the initial point $x^0\in\mathcal{T}_{\gamma}$ and the modified Polyak stepsizes $\alpha_k=\frac{\varphi(x^k)-\varphi^*}{2\|g(x^k)+h(x^k)\|^2}$ apply.
\end{itemize}
Then,
there exists $0<\mu<1$, in such a way that
\[
\dist^2(x^{k+1},\setX^*)
\leq\mu\dist^2(x^{k},\setX^*),
\]
holds for all iteration counts $k=0,1,2,...$.
\end{theorem}
\begin{proof}
\begin{itemize}
\item[{\rm(i)}]
\noindent{\rm (1)} By the subgradient upper bound {\rm(B$_{\mbox{\tiny A-D}}$)}, the term $T_1$ of the uniform recursion~\eqref{eq:recur_linear} satisfies that $T_1\leq 2\alpha_k^2c_9^2\dist^2(x^k,\setX^*)$. Thus, we have
\begin{align*}
\dist^2(x^{k+1},\setX^*)
\leq&\left(1+(\rho+2\eta)\alpha_k\right)\dist^2(x^k,\setX^*)-2\alpha_k\left[\varphi(x^k)-\varphi^*\right]+2\alpha_k^2c_9^2\dist^2(x^k,\setX^*)\nonumber\\
=&\left(1+(\rho+2\eta)\alpha_k+2\alpha_k^2c_9^2\right)\dist^2(x^k,\setX^*)-2\alpha_k\left[\varphi(x^k)-\varphi^*\right]\\
\leq&\left[1-\alpha_k\left(\frac{2}{\kappa_{gQG}^{\varphi}}-(\rho+2\eta)\right)+2\alpha_k^2c_9^2\right]\dist^2(x^{k},\setX^*)
\end{align*}
where the last inequality is because $\varphi$ satisfies the quadratic growth  property.
By setting $\alpha_k=\alpha\in\left(0,\frac{2-\kappa_{gQG}^{\varphi}(\rho+2\eta)}{2c_9^2\kappa_{gQG}^{\varphi}}\right)$, we have $1-\alpha_k\left(\frac{2}{\kappa_{gQG}^{\varphi}}-(\rho+2\eta)\right)+2\alpha_k^2c_9^2<1$. The desired result thus follows.

\noindent{\rm (2)} Under the quadratic growth condition and statement (iv) of~\Cref{prop:bad}, the subgradient upper bound condition {\rm(B$_{\mbox{\tiny {A-D}}}$)} with $c_{10}=0$ is equivalent to the subgradient upper bound condition {\rm(B$_{\mbox{\tiny {G}}}$)} with constant $c_{12}=0$. By the quadratic growth condition, if  {\rm(B$_{\mbox{\tiny {A-D}}}$)} with $c_{10}=0$ holds then  {\rm(B$_{\mbox{\tiny {G}}}$)} with $c_{12}=0$ holds.
Therefore,
the claimed result follows from the previous statement, with the adjustment made to the stepsizes.

\noindent{\rm (3)} Since $\eta=0$ and $\alpha_k=\frac{\varphi(x^k)-\varphi^*}{2\|g(x^k)+h(x^k)\|^2}$, the uniform recursion~\eqref{eq:recur_linear} yields that
\begin{equation}
\dist^2(x^{k+1},\setX^*)
\leq\left(1+\rho\alpha_k\right)\dist^2(x^k,\setX^*)-\alpha_k\left[\varphi(x^k)-\varphi^*\right].\label{eq:linear_2}
\end{equation}
Now, $\varphi$ satisfies the quadratic growth  property with $\kappa_{\tiny gQG}^{\varphi}$. Therefore,  by~\eqref{eq:linear_2} it follows that
\begin{equation}\label{eq:linear_3}
\dist^2(x^{k+1},\setX^*)
\leq\dist^2(x^{k},\setX^*)-\alpha_k(1-\rho\kappa_{\tiny gQG}^{\varphi})\left[\varphi(x^k)-\varphi^*\right].
\end{equation}
Furthermore, $\kappa_{\tiny gQG}^{\varphi}<\frac{1}{\rho}$, we obtain $\dist^2(x^{k+1},\setX^*)
\leq\dist^2(x^{k},\setX^*)\leq\dist^2(x^{0},\setX^*)$. Thus, the sequence $\{x^k\}$ generated by the algorithm is bounded and $\lim_{k\rightarrow\infty}\left[\varphi(x^k)-\varphi^*\right]=0$.
Since {\rm (B$_6$)} holds on $\mathcal{T}_{\gamma_0}$, we have a positive number $c_8>0$ such that $\|g(x^k)+h(x^k)\|\leq c_8$. Then, $\alpha_k=\frac{\varphi(x^k)-\varphi^*}{2\|g(x^k)+h(x^k)\|^2}\geq\frac{\varphi(x^k)-\varphi^*}{2c_8^2}$, and~\eqref{eq:linear_3} follows that
\begin{equation}\label{eq:linear_4}
\dist^2(x^{k+1},\setX^*)
\leq\dist^2(x^{k},\setX^*)-\frac{1-\rho\kappa_{\tiny gQG}^{\varphi}}{2c_8^2}\left[\varphi(x^k)-\varphi^*\right]^2.
\end{equation}
By $\lim_{k\rightarrow\infty}\left[\varphi(x^k)-\varphi^*\right]=0$ and the global QG of $\varphi$, there exists $k_0>0$ such that for all $k>k_0$, $\dist(x^k,\setX^*)<\gamma$ with some $\gamma>0$. By $\varphi$ is $\mathcal{T}_{\gamma}$-sharpness with constant $\kappa_{\tiny\rm s}^{\varphi}>0$,  then~\eqref{eq:linear_4} follows that
\begin{equation}\label{eq:linear_5}
\dist^2(x^{k+1},\setX^*)
\leq\left(1-\frac{(1-\rho\kappa_{\tiny gQG}^{\varphi})}{2c_8^2(\kappa_{\tiny\rm s}^{\varphi})^{2}}\right)\dist^{2}(x^{k},\setX^*),\quad\forall k>k_0.
\end{equation}
For all $k\leq k_0$, we have $\varphi(x^k)-\varphi^*\geq\frac{1}{\kappa_{\tiny gQG}^{\varphi}}\dist^2(x^k,\setX^*)\geq\frac{\gamma^2}{\kappa_{\tiny gQG}^{\varphi}}$, then $\left[\varphi(x^k)-\varphi^*\right]^2\geq\frac{\gamma^2}{\kappa_{\tiny gQG}^{\varphi}}\left[\varphi(x^k)-\varphi^*\right]\geq\frac{\gamma^2}{(\kappa_{\tiny gQG}^{\varphi})^2}\dist^2(x^k,\setX^*)$. Therefore,~\eqref{eq:linear_4} follows that
\begin{equation}\label{eq:linear_6}
\dist^2(x^{k+1},\setX^*)
\leq\left(1-\frac{(1-\rho\kappa_{\tiny gQG}^{\varphi})\gamma^2}{2c_8^2(\kappa_{\tiny gQG}^{\varphi})^2}\right)\dist^2(x^{k},\setX^*),\quad\forall k\leq k_0.
\end{equation}
\item[{\rm(ii)}] By the $\kappa_{\rm\tiny s}^{\varphi}$-sharpness of $\varphi$ on $\mathcal{T}_{\gamma}$ and $\eta=0$, for any $x^k\in\mathcal{T}_{\gamma}$, the uniform recursion~\eqref{eq:recur_linear} yields
\begin{align}
\dist^2(x^{k+1},\setX^*)
\leq&\left(1+\rho\alpha_k\right)\dist^2(x^k,\setX^*)-\alpha_k\left[\varphi(x^k)-\varphi^*\right]\nonumber\\
\leq&\dist^2(x^k,\setX^*)-\alpha_k\left(\frac{1}{\kappa_{\rm\tiny s}^{\varphi}}-\rho\dist(x^k,\setX^*)\right)\dist(x^k,\setX^*).\label{eq:linear_7}
\end{align}
Because $x^k\in\mathcal{T}_{\gamma}$ with $\gamma<\frac{1}{\rho\kappa_{\rm\tiny s}^{\varphi}}$, we have $\frac{1}{\kappa_{\rm\tiny s}^{\varphi}}-\rho\dist(x^k,\setX^*)>0$. Now that {\rm(B$_6$)} holds on $\mathcal{T}_{\gamma}$, there exists $c_8>0$ such that $\|g(x^k)+h(x^k)\|\leq c_8$. Furthermore, in combination with the $\kappa_{\rm\tiny s}^{\varphi}$-sharpness of $\varphi$ it leads to  $\alpha_k=\frac{\varphi(x^k)-\varphi^*}{2\|g(x^k)+h(x^k)\|^2}\geq\frac{\dist(x^k,\setX^*)}{2c_8^2\kappa_{\rm\tiny s}^{\varphi}}$. Then~\eqref{eq:linear_7} yields that
\[
\dist^2(x^{k+1},\setX^*)
\leq\left(1-\frac{1-\rho\kappa_{\rm\tiny s}^{\varphi}\gamma}{2c_8^2(\kappa_{\rm\tiny s}^{\varphi})^2}\right)\dist^2(x^k,\setX^*),
\]
and $x^{k+1}\in\mathcal{T}_{\gamma}$. The desired linear convergence result follows.
\end{itemize}
\end{proof}
\begin{remark}
{\rm We noted that the QG condition $\kappa_{gQG}^{\varphi}>\rho$ (or $\kappa_{gQG}^{\varphi}>\rho+2\eta$) does not mean that $\varphi$ is convex. We can easily construct the weakly convex function satisfy the conidions of~\Cref{cor:qg_wcvx_2}.}
\end{remark}
\begin{remark}
{\rm In contrast with the statement (ii) of~\Cref{cor:qg_wcvx_2}, by using the uniform recursion~~\eqref{eq:recur_linear}, we also can provide the linear convergence of Proj-SubGrad method with the Polyak style stepsize $\alpha_k=\frac{\varphi(x^k)-\varphi^*}{\|g(x^k)+h(x^k)\|^2}$  under the global sharpness condition (i.e., $\varphi(x)-\varphi^*\geq\kappa_{\mbox{\rm\tiny s}}\dist(x,\setX^*)$, $\forall x\in\setX$). For a complete proof, cf.~\cite{dima2018subgradient}).
}
\end{remark}
\section{Extention to stochastic subgradient methods} \label{sec:sgd}
In this section, we consider a
nonsmooth stochastic optimization problem and turn our attention to analyzing the subgradient upper bounds of nonsmooth stochastic optimization and giving more general convergence rate for the  proximal stochastic subgradient method. We note that the normalizing stochastic subgradients (stepsize as (B$_1$)) may introduce bias or may not be well-defined if $g(x^k;\xi)=0$, so it is hard to adapt subgradient upper bound scenario (B$_1$) to stochastic optimization. Therefore, the expected stochastic subgradient upper bound of term $f$ is the key issue for nonsmooth stochastic optimization.
We shall analyze adaptive (B$_{\mbox{\tiny {D-D}}}$),  (B$_2$), and (B$_4$) for the following stochastic optimization model:
\begin{equation}\label{eq:stchastic_opt problem}
\min_{x\in \setX} \ {\varphi}(x)=f(x)+r(x)=\setE_{\xi\sim\Xi}\left[{f}(x,\xi)\right]+r(x),
\end{equation}
where $\xi$ is
a random vector,
which is assumed to follow the probability distribution $\Xi$; $f(\cdot,\xi):\setX\rightarrow\R$ and $r:\setX\rightarrow\R$ are assumed to be proper and lower-semicontinuous. We shall also denote $\setX^*$ to be the set of optimal solutions of~\eqref{eq:stchastic_opt problem}. We assume $\setX^*\neq\emptyset$. 
For any $x^*\in\setX^*$, we denote $\varphi^*=\varphi(x^*)$. If problem \eqref{eq:stchastic_opt problem} is nonconvex, then the set of its critical points is denoted by $\overline \setX$. Moreover, we consider solving \eqref{eq:stchastic_opt problem} via the following stochastic subgradient method:
\begin{algorithm}[ht]
	\caption{Sto-SubGrad: Stochastic Subgradient method for solving \eqref{eq:stchastic_opt problem}}
	{\bf Initialization:}  $x^0$ and $\alpha_0$;
	\begin{algorithmic}[1]
		\item[1:] {\bf for} {$k=0,1,\ldots$} {\bf do}
		\item[2:] $\quad$Choose $\xi_k\sim\Xi$ and compute a subgradient $g(x^k,\xi_k)\in\partial f(x^k,\xi_k)$;
		\item[3:] $\quad$Update the step size $\alpha_{k}$ according to a certain rule;
		\item[4:] $\quad$Update $x^{k+1}: =\arg\min_{x\in\setX} \langle g(x^k,\xi_k),x\rangle+r(x)+\frac{1}{2\alpha_k}\|x-x^{k}\|^2$.
		\item[5:] {\bf end for}
	\end{algorithmic}
	\label{alg:SGD}
\end{algorithm}
\subsection{Expected stochastic subgradient upper bounds}\label{sec:sto_subgrad_upp}
Before analyzing Sto-SubGrad, let us introduce the following expected stochastic subgradient upper bounding condition which guarantees the convergence of Sto-SubGrad without assuming Lipschitz continuity of the objective function $f(x)$. These upper bounds can be seen as adaptive subgradient upper bounds (B$_2$)-(B$_4$) for the stochastic optimization problem. The following expected subgradient upper bound (B$_{\mbox{\tiny D-D}}$) was used in~\cite{davis2019stochastic}; (B$_2'$) was used in~\cite{culioli1990}. However, conditions (B$_3'$) and (B$_4'$) are new. For given $\xi\sim\Xi$, let $g(x,\xi)$ be one subgradient of $f(x,\xi)$, $\forall x\in\setX$.

\begin{definition}[Expected stochastic subgradient upper bounds] {\ }
{\rm
\begin{itemize}
\item[{\rm (B$_{\mbox{\tiny D-D}}'$).}]
$\varphi$ is said to have the $f$-expected Lipschitz subgradient upper bound if for all $x\in\setX$, $\forall g(x,\xi)\in\partial f(x,\xi)$, $\|g(x,\xi)\|\leq L_f(\xi)$ and $\setE_{\xi\sim\Xi}[(L_f(\xi))^2]\leq L_f^2$ with $L_f\geq0$.
\item[{\rm (B$_2'$).}]
$\varphi$ is said to have the linear expected subgradient upper bound if there are $c_1,c_2\geq0$ such that for all $x\in\setX$,  $\forall g(x,\xi)\in\partial f(x,\xi)$, $\forall h(x)\in\partial r(x)$, it holds that $\setE_{\xi\sim\Xi}\left[\|g(x,\xi)+h(x)\|^2\right]\leq c_1\|x\|^2+c_2$.
\item[{\rm (B$_3'$).}]
$\varphi$ is said to have the $[\varphi_{\lambda}(x)-\varphi^*]$-expected subgradient upper bound if there are $c_3,c_4\geq0$ such that for all $x\in\setX$, $\forall g(x,\xi)\in\partial f(x,\xi)$, $\forall h(x)\in\partial r(x)$, it holds that $\setE_{\xi\sim\Xi}\left[\|g(x,\xi)+h(x)\|^2\right]\leq c_3[\varphi_{\lambda}(x)-\varphi^*]+c_4$.
\item[{\rm (B$_4'$).}]
$\varphi$ is said to to have the $\|\nabla \varphi_{\lambda}(x)\|$-expected subgradient upper bound if there are $c_5, c_{6}\geq0$ such that for all $x\in\setX$, $\forall g(x,\xi)\in\partial f(x,\xi)$, $\forall h(x)\in\partial r(x)$, it holds that
$\setE_{\xi\sim\Xi}\left[\|g(x,\xi)+h(x)\|^2\right]\leq c_5\|\nabla\varphi_{\lambda}(x)\|^2+c_{6}$.
\end{itemize}
}
\end{definition}

Next, we note two expected stochastic subgradient upper bounding conditions (B$_{\mbox{\tiny {A-D}}}'$) and (B$_{\mbox{\tiny {G}}}'$) from~\cite{asi2019importance} and~\cite{grimmer2018} respectively.

\begin{definition}{\ }

{\rm
\begin{itemize}
\item[{\rm (B$_{\mbox{\tiny {A-D}}}'$).}]
$\partial\varphi$ is said to have $\dist(x,\setX^*)$-expected stochastic subgradient upper bound if there are $c_9, c_{10}\geq0$ such that for all $x\in\setX$, $\forall g(x)\in \partial f(x)$ and $\forall h(x)\in \partial r(x)$, it holds that $\setE_{\xi\sim\Xi}\left[\|g(x,\xi)+h(x)\|^2\right]\leq c_9\dist^2(x,\setX^*)+c_{10}$.
\item[{\rm (B$_{\mbox{\tiny {G}}}'$).}]
$\partial\varphi$ is said to have $[\varphi(x)-\varphi^*]$-expected stochastic subgradient upper bound if $c_{11}, c_{12}\geq0$ such that for all $x\in\setX$, $\forall g(x)\in \partial f(x)$ and $\forall h(x)\in \partial r(x)$, it holds that $\setE_{\xi\sim\Xi}\|g(x,\xi)+h(x)\|^2\leq c_{11}[\varphi(x)-\varphi^*]+c_{12}$.
\end{itemize}
}
\end{definition}
\subsubsection{Relationship among stochastic subgradient upper bounds}
The following two propositions illustrate the relationships among the notions of subgradient upper bounding conditions (B$_2'$)-(B$_4'$) when $f(\cdot,\xi)$ is $\rho(\xi)$-weakly convex. The technical proofs of the following~\Cref{prop:upperbound6},~\Cref{prop:upperbound7}, and~\Cref{prop:bad-s} are similar to that of~\Cref{prop:upperbound2},~\Cref{prop:upperbound3}, and~\Cref{prop:bad}, respectively. For the sake of brevity, we omit the proofs here.
\begin{proposition}\label{prop:upperbound6}{\bf(Relationship between expected stochastic subgradient upper bounds in weakly convex case)} Suppose that $r$ is $\eta$-weakly convex with $\eta\geq0$ and $f(\cdot,\xi)$ is $\rho(\xi)$-weakly convex with $\rho(\xi)>0$ and $\setE_{\xi\in\Xi}\left[\rho(\xi)\right]\leq\rho$ with $\rho$ being a positive number. Then, we have $\mbox{\rm(B$_4'$) $\Longrightarrow$ (B$_3'$)  $\Longrightarrow$ (B$_2'$).}$
\end{proposition}
\begin{proposition}\label{prop:upperbound7} Suppose $r$ is $\eta$-weakly convex with $\eta\geq0$ and $f(\cdot,\xi)$ is $\rho(\xi)$-weakly convex with $\rho(\xi)>0$ and $\setE_{\xi\in\Xi}\left[\rho(\xi)\right]\leq\rho$ with $\rho$ being a positive number. Then, the following assertions hold:
\begin{itemize}
\item[{\rm (i)}] If $\varphi$ satisfies the global quadratic growth property with parameter $\kappa_{gQG}^{\varphi}>0$, i.e., $\dist^2(x,\setX^*)\leq\kappa_{gQG}^{\varphi}[\varphi(x)-\varphi^*]$, $\forall x\in\setX$, and the set of critical points $\overline{ {\setX} }$ is bounded by constant $\sB_{\overline{\setX}}>0$, then we have that {\rm(B$_2'$) $\Longrightarrow$ (B$_3'$)}. In other words, upper bounds {\rm(B$_2'$)-(B$_3'$)} are equivalent in this case.
\item[{\rm (ii)}] If $\partial\varphi$ satisfies the global metric subregularity property with parameter $\kappa_{gMS}^{\varphi}>0$, i.e., $\dist(x,\overline{\setX})\leq\kappa_5\dist(0,\partial\varphi(x))$, $\forall x\in\setX$, and the set of critical points $\overline{ {\setX} }$ is bounded by constant $\sB_{\overline{\setX}}>0$, then we have that {\rm(B$_2'$) $\Longrightarrow$ (B$_4'$)}. In other words, upper bounds {\rm(B$_2'$) and (B$_4'$)} are equivalent in this case.
\item[{\rm (iii)}] If there are $\tau$  ($0<\tau<1$) and $R_0>0$ such that $\|\prox_{\lambda,\varphi}(x)\|\leq \tau \|x\|,\quad\forall x\in\{x\in\setX\mid\|x\|\geq R_0\}$, then {\rm(B$_2'$) $\Longrightarrow$ (B$_4'$)}. In other words, upper bounds {\rm(B$_2'$) and (B$_4'$)} are equivalent in this case.
\end{itemize}
\end{proposition}
The following~\Cref{prop:bad-s} will establish the relationships among (B$_{\mbox{\tiny {A-D}}}'$), (B$_{\mbox{\tiny {G}}}'$), and (B$_2'$)-(B$_4'$).
\begin{proposition}\label{prop:bad-s}
Suppose $r$ is $\eta$-weakly convex with $\eta\geq0$ and $f(\cdot,\xi)$ is $\rho(\xi)$-weakly convex with $\rho(\xi)>0$ and $\setE_{\xi\in\Xi}\left[\rho(\xi)\right]\leq\rho$ with $\rho$ being a positive number. $\varphi_{\lambda}$ is the Moreau envelope of $\varphi$ with $\lambda<\frac{1}{\rho+\eta}$. Then, we have
\begin{itemize}
\item[{\rm (i)}] {\rm(B$_{\mbox{\tiny G}}'$)} $\Longleftrightarrow$ {\rm(B$_3'$)}.
\item[{\rm (ii)}] {\rm(B$_{\mbox{\tiny {A-D}}}'$)} $\Longrightarrow$ {\rm(B$_2'$)}. Conversely, if the optimal solution set $\setX^*$ is bounded, then {\rm (B$_2'$)} $\Longrightarrow$ {\rm(B$_{\mbox{\tiny {A-D}}}'$)}.
\item[{\rm (iii)}] {\rm(B$_G'$)} $\Longrightarrow$ {\rm(B$_{\mbox{\tiny {A-D}}}'$)} with $c_9=\sqrt{c_{11}(\rho+\eta+c_{11})}$ and $c_{10}=\sqrt{2c_{12}}$. Conversely, if $\varphi$ satisfies the global quadratic growth property with parameter $\kappa_{gQG}^{\varphi}>0$, i.e., $\dist^2(x,\setX^*)\leq\kappa_{gQG}^{\varphi}[\varphi(x)-\varphi^*]$, $\forall x\in\R^d$, then we have {\rm(B$_{\mbox{\tiny {A-D}}}'$)} $\Longrightarrow$ {\rm(B$_3'$)}.
\end{itemize}
\end{proposition}
\subsection{Convergence analysis of Sto-SubGrad for stochastic weakly convex optimization}\label{sec:rate-SGD}
At the beginning of this subsection, we establish the basic recursion of Sto-SubGrad for the weakly convex case. The technical proofs of the following~\Cref{lemma:inequality_iter_s},~\Cref{lemma:basic_recur_wcvx_hx_s}, and~\Cref{lemma:profound_recur_wcvx_s} are similar to~\Cref{lemma:basic_recur_wcvx},~\Cref{lemma:basic_recur_wcvx_hx}, and~\Cref{lemma:profound_recur_wcvx}, respectively. For the sake of brevity, we omit the proofs.
\begin{lemma}[Basic recursion of Sto-SubGrad]\label{lemma:inequality_iter_s}
Suppose that $r$ is $\eta$-weakly convex with $\eta>0$, $f(\cdot,\xi)$ is $\rho(\xi)$-weakly convex with $\rho(\xi)>0$ and $\setE_{\xi\sim\Xi}\left[\rho(\xi)\right]\leq\rho$ with $\rho$ and $\lambda\leq\frac{1}{\rho+2\eta}$ be two positive numbers. Let $\{x^k\}_{k\in \N}$ be the sequence generated by Sto-SubGrad  for solving problem \eqref{eq:stchastic_opt problem}. Then, for all $x\in\setX$, the following hold:
\begin{itemize}
\item[{\rm(i)}] If $\alpha_k<\frac{1}{6\eta}$, it holds that
\begin{center}
$
\begin{aligned}
\mathbb{E}_{\xi_k}\|x-x^{k+1}\|^2
\leq&\left(1-\frac{(1-\lambda(\rho+2\eta))\alpha_k}{\lambda}\right)\|x-x^k\|^2-2\alpha_k\left[\varphi(x^k)-\varphi(x)-\frac{1}{2\lambda}\|x-x^k\|^2\right] \\
&+2\alpha_k^2\mathbb{E}_{\xi_k}\|g(x^k,\xi_k)+h(x^k)\|^2.
\end{aligned}
$
\end{center}
\item[{\rm(ii)}] Similarly, if $\alpha_k<\frac{1}{2(\rho+2\eta+\frac{1}{\lambda})}$, then it holds that
\begin{center}
$
\begin{aligned}
\mathbb{E}_{\xi_k}\|x-x^{k+1}\|^2
\leq&\left(1-\frac{(1-\lambda(\rho+2\eta))\alpha_k}{\lambda}\right)\|x-x^k\|^2
-2\alpha_k\mathbb{E}_{\xi_k}\left[\varphi(x^{k+1})-\varphi(x)-\frac{1}{2\lambda}\|x-x^k\|^2\right] \\
 &-\frac{\alpha_k}{\lambda}\mathbb{E}_{\xi_k}\|x^k-x^{k+1}\|^2
+2\alpha_k^2\mathbb{E}_{\xi_k}\|g(x^k,\xi_k)-g(x^{k+1},\xi_k)\|^2.
\end{aligned}
$
\end{center}
\item[{\rm(iii)}] If $\alpha_k\leq\frac{1}{2\eta}$, then $\|x^k-x^{k+1}\|\leq\alpha_k\|g(x^k,\xi_k)+h(x^k)\|$.
\end{itemize}
\end{lemma}
\begin{lemma}[Recursion of Sto-SubGrad with $x=\hat{x}^k=\prox_{\lambda,\varphi}(x^k)$ and $x=x^*$]\label{lemma:basic_recur_wcvx_hx_s} Suppose that $r$ is $\eta$-weakly convex with $\eta>0$, $f(\cdot,\xi)$ is $\rho(\xi)$-weakly convex with $\rho(\xi)>0$ and $\setE_{\xi\sim\Xi}\left[\rho(\xi)\right]\leq\rho$ with $\rho$ and $\lambda\leq\frac{1}{\rho+2\eta}$ be two positive numbers. Let $\{x^k\}_{k\in \N}$ be the sequence generated by Sto-SubGrad  for solving problem \eqref{eq:stchastic_opt problem}. Then, the following hold:
\begin{itemize}
\item[{\rm(i)}] If $\alpha_k<\frac{1}{6\eta}$, it holds that
\begin{center}
$\mathbb{E}_{\xi_k}\|\hat{x}^k-x^{k+1}\|^2\leq\left(1-\frac{(1-\lambda(\rho+2\eta))\alpha_k}{\lambda}\right)\|\hat{x}^k-x^k\|^2+\underbrace{2\alpha_k^2\mathbb{E}_{\xi_k}\|g(x^k,\xi_k)+h(x^k)\|^2}_{T_1'}$.
\end{center}
\item[{\rm(ii)}] Similarly, it holds that
\begin{center}
$\mathbb{E}_{\xi_k}\|\hat{x}^k-x^{k+1}\|^2\leq\left(1-\frac{\left(1-\lambda\left(\rho+2\eta\right)\right)\alpha_k}{\lambda}\right)\|\hat{x}^k-x^k\|^2+\underbrace{2\alpha_k^2\mathbb{E}_{\xi_k}\|g(x^k,\xi_k)-g(x^{k+1},\xi_k)\|^2}_{T_2'}$.
\end{center}
\item[{\rm(iii)}] If $\alpha_k<\frac{1}{6\eta}$, it holds that
\begin{center}
$
\begin{aligned}
\mathbb{E}_{\xi_k}\dist^2(x^{k+1},\setX^*)\leq&(1+(\rho+2\eta)\alpha_k)\dist^2(x^k,\setX^*)\\
&-2\alpha_k\left[\varphi(x^k)-\varphi^*\right]+\underbrace{2\alpha_k^2\mathbb{E}_{\xi_k}\|g(x^k,\xi_k)+h(x^k)\|^2}_{T_1'}
\end{aligned}
$
\end{center}
\end{itemize}
\end{lemma}
\begin{lemma}\label{lemma:profound_recur_wcvx_s}{\bf(Moreau envelope recursive of Sto-SubGrad relations 
under various subgradient upper bounding conditions)}
Suppose that $r$ is $\eta$-weakly convex with $\eta>0$, $f(\cdot,\xi)$ is $\rho(\xi)$-weakly convex with $\rho(\xi)>0$ and $\setE_{\xi\sim\Xi}\left[\rho(\xi)\right]\leq\rho$ with $\rho$ and $\lambda\leq\frac{1}{\rho+2\eta}$ be two positive numbers. Let $\{x^k\}_{k\in \N}$ be the sequence generated by Sto-SubGrad  for solving problem \eqref{eq:stchastic_opt problem}. Then, for the subgradient upper bounds scenarios {\rm(B$_{\mbox{\tiny D-D}}'$)} and {\rm(B$_3'$)}, we uniformly have the recursion with $\gamma_4, \gamma_5>0$ as
\be\label{eq:lemma_sto_6}
\mathbb{E}_{\xi_k}\left[\varphi_{\lambda}(x^{k+1})-\varphi^*\right]-\left(1+\gamma_4\alpha_k^2\right)[\varphi_{\lambda}(x^k)-\varphi^*]
\leq-\frac{(1-\lambda(\rho+2\eta))\alpha_k}{2}\|\nabla\varphi_{\lambda}({x}^k)\|^2+\gamma_5\alpha_k^2.
\ee
\end{lemma}
Finally, we establish the convergence and convergence rate of Sto-SubGrad when $r$ and $f(\cdot,\xi)$ are weakly convex and one of the stochastic subgradient upper bounding scenarios (B$_{\mbox{\tiny D-D}}'$) and (B$_3'$)-(B$_4'$) holds.

\begin{theorem}[Convergence analysis of Sto-SubGrad]
\label{theo:sto_convergence_wcvx} Suppose $r$ is $\eta$-weakly convex with $\eta\geq0$ and $f(\cdot,\xi)$ is $\rho(\xi)$-weakly convex with $\rho\geq\setE_{\xi\sim\Xi}\rho(\xi)\geq0$. Let $\{x^k\}_{k\in \N}$ be the sequence generated by Sto-SubGrad for solving problem \eqref{eq:stchastic_opt problem}. If one of stochastic subgradient upper bounds \{{\rm(B$_{\mbox{\tiny D-D}}'$), (B$_3'$), (B$_4'$)}\} holds, and the stepsize sequence $\{\alpha_k\}$ is a $\sigma$-sequence: $\alpha_k>0$, $\sum_{k\in \N} \alpha_k=\infty$, and $\sum_{k\in \N}\alpha_k^2\leq\bar{\mathfrak{a}}<\infty$, and $\lambda$ satisfies $\lambda<\frac{1}{\rho+2\eta}$, then the following statements hold true:
\begin{itemize}
\item[{\rm (i)}] If $\varphi$ is level bounded, then $\lim_{k\rightarrow\infty}\|\nabla\varphi_{\lambda}(x^k)\|=0$ almost surely and hence, every accumulation point of $\{x^k\}_{k\in \N}$ is a critical point of problem \eqref{eq:stchastic_opt problem} almost surely.
\item[{\rm (ii)}] $\min\limits_{0\leq k\leq T}\mathbb{E}_{\xi_{k-1}}\left[\|\nabla \varphi_{\lambda}({x}^k)\|^2\right]\leq\left(\frac{2}{1-\lambda(\rho+2\eta)}\right)\frac{\left[\varphi_{\lambda}(x^0)-\varphi^*\right]+\gamma_6\sum_{k=0}^T\alpha_k^2}{\sum_{k=0}^T\alpha_k}$\\
with $\gamma_6\geq\gamma_5+\gamma_4\left[\varphi_{\lambda}(x^0)-\varphi^*+\gamma_5\bar{\mathfrak{a}}\right]e^{1+\gamma_4\bar{\mathfrak{a}}}$. If $\alpha_k=\frac{\Delta_4}{\sqrt{T+1}}$, $k=0,1,...,T$, with $\Delta_4\geq0$, then $\min\limits_{0\leq k\leq T}\mathbb{E}_{\xi_{k-1}}\left[\|\nabla\varphi_{\lambda}({x}^k)\|^2\right]\leq R_3/\sqrt{T+1}$, with $R_3\geq0$.
\end{itemize}
\end{theorem}
\begin{proof}
\noindent{\rm(i)} Recall \eqref{eq:lemma_sto_6} in~\Cref{lemma:profound_recur_wcvx_s}:
\be\label{eq:dimRS_wcvx_2}
\mathbb{E}_{\xi_k}\left[\varphi_{\lambda}(x^{k+1})-\varphi^*\right]-\left(1+\gamma_4\alpha_k^2\right)[\varphi_{\lambda}(x^k)-\varphi^*]\leq-\frac{(1-\lambda(\rho+2\eta))\alpha_k}{2}\|\nabla\varphi_{\lambda}({x}^k)\|^2+\gamma_5\alpha_k^2.
\ee
Note that $\varphi_\lambda(x)-\varphi^*\geq 0$, $\|\nabla\varphi_{\lambda}(x^k)\|^2\geq0$, and $\sum_{k=0}^{\infty}\alpha_k^2<\infty$. Applying \Cref{thm:RS} on~\eqref{eq:dimRS_wcvx_2} yields that $\lim_{k\rightarrow\infty}\left(\varphi_\lambda(x^k)-\varphi^*\right)$ almost surely exists and is finite. By the definition of $\varphi_\lambda$, we have $\{\varphi(\prox_{\lambda,\varphi}(x^k)) + \frac{1}{2\lambda} \|x^k - \prox_{\lambda,\varphi}(x^k)\|^2 -\varphi^*\}_{k\in \N}$ is almost surely bounded. By the level boundedness of $\varphi$, we obtain that $\{x^k\}_{k\in \N}$ is almost surely bounded.

Next, combining the weakly convexity of $f$ and $r$, the almost surely boundedness of $\{x^k\}$, and statement (iii) of~\Cref{lemma:inequality_iter_s} leads to the event $\Omega_1$ defined as follows:
\[
\mathbb{P}\left(\left\{w\in\Omega_1:\|x^k(w)-x^{k+1}(w)\|\leq M'(w)\alpha_k, M'(w)>0\right\}\right)=1.
\]
Applying \Cref{thm:RS} and~\eqref{eq:dimRS_wcvx_2} also gives rise to event $\Omega_2$ as follows:
\[
\mathbb{P}\left(\left\{w\in\Omega_2: \sum_{k=0}^{\infty}\alpha_k\|\nabla\varphi_{\lambda}(x^k(w))\|^2<\infty\right\}\right)=1.
\]
Note that $\sum_{k=0}^{\infty}\alpha_k=\infty$ and $\nabla\varphi_{\lambda}$ is Lipschitz continuous. By~\Cref{lemma:rcs}, we have
\[
\lim_{k\rightarrow\infty}\|\nabla\varphi_{\lambda}(x^k(w))\|=0, \, \forall w\in \Omega_1 \cap \Omega_2.
\]
It is easy to see that $\mathbb{P}\left(\Omega_1\cap\Omega_2\right)=1$. Hence, we have $\lim_{k\rightarrow\infty}\|\nabla\varphi_{\lambda}(x^k)\|=0$ almost surely. Invoking \Cref{cor:mecritical}, we conclude that every accumulation point of $\{x^k\}_{k\in \N}$ is a critical point of~\eqref{eq:stchastic_opt problem} almost surely.

\noindent{\rm(ii)} Taking expectation $\mathbb{E}_{\xi_{k-1}}$ of~\eqref{eq:lemma_sto_6} in~\Cref{lemma:profound_recur_wcvx_s}, we obtain
\be\label{eq:rate_sto_wcvx_1}
\begin{aligned}
&\mathbb{E}_{\xi_k}\left[\varphi_{\lambda}(x^{k+1})-\varphi^*\right]-\left(1+\gamma_4\alpha_k^2\right)\mathbb{E}_{\xi_{k-1}}[\varphi_{\lambda}(x^k)-\varphi^*]\\
\leq&-\frac{(1-\lambda(\rho+2\eta))\alpha_k}{2}\mathbb{E}_{\xi_{k-1}}\|\nabla\varphi_{\lambda}({x}^k)\|^2+\gamma_5\alpha_k^2.
\end{aligned}
\ee
Following Statement (ii) in~\Cref{lemma:Polyak_extend} and~\eqref{eq:rate_sto_wcvx_1}, the term $\mathbb{E}_{\xi_{k-1}}\left[\varphi_{\lambda}(x^{k})-\varphi^*\right]$ can be upper bounded by the positive number $\left[\varphi_{\lambda}(x^0)-\varphi^*+\gamma_5\bar{\mathfrak{a}}\right]e^{1+\gamma_4\bar{\mathfrak{a}}}$. Observe further that~\eqref{eq:rate_sto_wcvx_1} gives rise to
\be\label{eq:rate_sto_wcvx_2}
\begin{aligned}
\mathbb{E}_{\xi_k}\left[\varphi_{\lambda}(x^{k+1})-\varphi^*\right]-\mathbb{E}_{\xi_{k-1}}[\varphi_{\lambda}(x^k)-\varphi^*]\leq-\frac{(1-\lambda(\rho+2\eta))\alpha_k}{2}\mathbb{E}_{\xi_{k-1}}\|\nabla\varphi_{\lambda}({x}^k)\|^2+\gamma_6\alpha_k^2
\end{aligned}
\ee
with $\gamma_6\geq\gamma_5+\gamma_4\left[\varphi_{\lambda}(x^0)-\varphi^*+\gamma_5\bar{\mathfrak{a}}\right]e^{1+\gamma_4\bar{\mathfrak{a}}}$.
Recursive relation \eqref{eq:rate_sto_wcvx_2} yields
\[
\min_{0\leq k\leq T}\mathbb{E}_{\xi_{k-1}}\left[\|\nabla \varphi_{\lambda}({x}^k)\|^2\right]\leq\left(\frac{2}{1-\lambda(\rho+2\eta)}\right)\frac{\left[\varphi_{\lambda}(x^0)-\varphi^*\right]+\gamma_6\sum_{k=0}^T\alpha_k^2}{\sum_{k=0}^T\alpha_k}.
\]
If $\alpha_k=\frac{\Delta_4}{\sqrt{T+1}}$, $k=0,1,...,T$, we obtain that $
\min_{0\leq k\leq T}\mathbb{E}_{\xi_{k-1}}\left[\|\nabla \varphi_{\lambda}({x}^k)\|^2\right]\leq\frac{R_3}{\sqrt{T+1}}$.
\end{proof}
\begin{corollary}\label{cor:sto_G}
Suppose the assumptions of~\Cref{theo:sto_convergence_wcvx} hold. Then, for any given positive non-decreasing function $G:\R_+\rightarrow\R_+$ with the property that $G(t)=t$ for $0\leq t<\delta$ and $G(t)>t$ for $t\geq\delta$ with $\delta>0$, we always have $\min\limits_{0\leq k\leq T}\mathbb{E}_{\xi_{k-1}}\left[\|\nabla \varphi_{\lambda}({x}^k)\|^2\right]\leq G(R_3/\sqrt{T+1})$.
\end{corollary}
Furthermore, an additional global KL property with exponent 1/2 leads to an $O(1/k)$-rate when $f(\cdot,\xi)$ and $r$ are weakly convex, as shown below.
\begin{theorem}[Further convergence analysis of Sto-SubGrade] 
\label{theo:sto_convergence_wcvx_qg} Suppose that the assumptions of~\Cref{theo:sto_convergence_wcvx} hold. Moreover, suppose that $\varphi$ satisfies the global KL property with exponent $1/2$ and parameter $\kappa_{gKL}^{\varphi}>0$. Then, we have:
\begin{itemize}
\item[{\rm (i)}] $\lim_{k\rightarrow\infty}\varphi_{\lambda}(x^k)=\varphi^*$ holds almost surely.
\item[{\rm (ii)}] $\mathbb{E}_{\xi_{k-1}}\left[\varphi_{\lambda}(x^k)-\varphi^*\right]=O(1/k)$, $k\in\N$.
\end{itemize}
\end{theorem}
\begin{proof}
\noindent{\rm(i)} By~\Cref{prop:global_eb}, $\varphi$ satisfies the global KL property with exponent $1/2$ and parameter $\kappa_{gKL}^{\varphi}>0$ yields that $\varphi_{\lambda}$ satisfies the global KL inequality with parameter $\kappa_{gKL}^{\varphi_{\lambda}}>0$, we have $\varphi_{\lambda}(x)-\varphi^*\leq\kappa_{gKL}^{\varphi_{\lambda}}\|\nabla\varphi_{\lambda}(x)\|^2$. By Statement (i) of~\Cref{theo:sto_convergence_wcvx}, we conclude that $\lim_{k\rightarrow\infty}\|\nabla\varphi_{\lambda}(x^k)\|=0$ holds almost surely. Therefore, $\lim_{k\rightarrow\infty}\varphi_{\lambda}(x^k)=\varphi^*$.

\noindent{\rm(ii)} Since $\varphi_{\lambda}$ satisfies the global KL property with exponent $1/2$ and parameter $\kappa_{gKL}^{\varphi_{\lambda}}>0$, we have $\varphi_{\lambda}(x)-\varphi^*\leq\kappa_{gKL}^{\varphi_{\lambda}} \|\nabla\varphi_{\lambda}(x)\|^2$. The recursive relation~\eqref{eq:rate_sto_wcvx_2} yields $\mathbb{E}_{\xi_k}\left[\varphi_{\lambda}(x^{k+1})-\varphi^*\right]-\left(1-\frac{\alpha_k(1-\lambda(\rho+2\eta))}{2\kappa_{gKL}^{\varphi_{\lambda}}}\right)\mathbb{E}_{\xi_{k-1}}[\varphi_{\lambda}(x^k)-\varphi^*]\leq\gamma_6\alpha_k^2$.
Letting $\alpha_k=\frac{\Delta_5}{k}$ with $\Delta_5>\frac{4\kappa_{gKL}^{\varphi_{\lambda}}}{1-\lambda(\rho+2\eta)}$, we have $\mathbb{E}_{\xi_k}\left[\varphi_{\lambda}(x^{k+1})-\varphi^*\right]-\left(1-\frac{2}{k}\right)\mathbb{E}_{\xi_{k-1}}[\varphi_{\lambda}(x^k)-\varphi^*]\leq\frac{\gamma_6\Delta_5^2}{k^2}$. By~\Cref{lemma:Polyak1}, the desired result follows.
\end{proof}

In the next theorem, we are going to derive the linear convergence analysis of Sto-SubGrad. The proof of the following theorem is similar to the statement (i-1) and (i-2) of~\Cref{cor:qg_wcvx_2}. For the sake of brevity, we omit the proof of the following theorem.
\begin{theorem}[Linear convergence of Sto-SubGrad]\label{cor:qg_wcvx_2_sto} Suppose that $r$ is $\eta$-weakly convex with $\eta>0$, $f(\cdot,\xi)$ is $\rho(\xi)$-weakly convex with $\rho(\xi)>0$ and $\setE_{\xi\sim\Xi}\left[\rho(\xi)\right]\leq\rho$ with $\rho$ and $\lambda\leq\frac{1}{\rho+2\eta}$ be two positive numbers. Let $\{x^k\}_{k\in \N}$ be the sequence generated by Sto-SubGrad  for solving problem \eqref{eq:stchastic_opt problem}. Suppose further $\varphi$ satisfies the global quadratic growth  property with $\kappa_{gQG}^{\varphi}>0$.
If one of the following two conditions holds:
\begin{itemize}
\item[{\rm (i)}] $\kappa_{gQG}^{\varphi}<\frac{2}{\rho+2\eta}$, and the subgradient upper bound condition {\rm(B$_{\mbox{\tiny {A-D}}}'$)} holds with $c_{10}=0$, and we select the constant stepsize as $\alpha_k=\alpha\in\left(0,\min\left\{\frac{2-\kappa_{gQG}^{\varphi}(\rho+2\eta)}{2c_9^2\kappa_{gQG}^{\varphi}},\frac{1}{6\eta}\right\}\right)$.
\item[{\rm (ii)}] $\kappa_{gQG}^{\varphi}<\frac{2}{\rho+2\eta}$, and the subgradient upper bound condition {\rm(B$_{\mbox{\tiny {G}}}'$)} holds with $c_{12}=0$, and we select
 the constant stepsize as$\alpha_k=\alpha\in\left(0,\min\left\{\frac{2-\kappa_{gQG}^{\varphi}(\rho+2\eta)}{2c_{11}(\rho+\eta+c_{11})\kappa_{gQG}^{\varphi}},\frac{1}{6\eta}\right\}\right)$.
\end{itemize}
Then, the linear convergence rate holds:  $\mathbb{E}_{\xi_k}\dist^2(x^{k+1},\setX^*)
\leq\mu\mathbb{E}_{\xi_{k-1}}\dist^2(x^{k},\setX^*)$, with $\mu<1$.
\end{theorem}



{\small
\appendix
\section{Several technical lemmas regarding convergent sequences}
We now introduce some useful lemmas that enable our convergence analysis for Prox-SubGrad.
\begin{lemma}\label{lemma:Polyak_extend}
Suppose $\{u^k\}$ and $\{v^k\}$ be two sequences of positive numbers, and
\noindent{$u^{k+1}\leq(1+\mu_k)u^k-v^k+\nu_k,\quad\mu_k\geq0,\;\nu_k\geq0,\;\sum_{k=0}^{\infty}\mu_k\leq\bar{\mu}<\infty,\;\sum_{k=0}^{\infty}\nu_k\leq\bar{\nu}<\infty$.}

\noindent{Then, {\rm(i)} $u^k\rightarrow u\geq0$;}
\quad{\rm(ii)} $u^{k+1}\leq\left(u^0+\bar\nu\right)e^{1+\bar\mu}$.
\end{lemma}
\begin{proof}
\noindent{\rm(i)} See Lemma 2 of~\cite{polyak1987}.

\noindent{\rm(ii)} By $v^k\geq0$, we obtain $u^{k+1}\leq(1+\mu_k)u^k+\nu_k$. Let $\zeta^k=u^k\prod_{j=0}^{k-1}\left(1+\mu^j\right)^{-1}$ with $\zeta^0=u^0$. Then, multiplying $\prod_{j=0}^{k}\left(1+\mu^j\right)^{-1}$ (note that $\prod_{j=0}^{k}\left(1+\mu^j\right)^{-1}<1$) on both sides of the above inequality yields $\zeta^{k+1}\leq\zeta^k+\nu^k\prod_{j=0}^{k}\left(1+\mu^j\right)^{-1}\leq\zeta^k+\nu^k\leq\zeta^0+\sum_{j=0}^k\nu^j\leq\zeta^0+\bar{\nu}=u^0+\bar\nu$, which further implies
\be\label{eq:bound_sequence}
\begin{aligned}
u^{k+1}\leq\left(u^0+\bar{\nu}\right)\prod_{j=0}^{k}\left(1+\mu^j\right).
\end{aligned}
\ee
Next, we shall upper bound the term $\prod_{j=0}^{k}\left(1+\mu^j\right)$, because $\prod_{j=0}^{k}\left(1+\mu^j\right)=e^{\log\prod_{j=0}^{k}\left(1+\mu^j\right)}=e^{\sum_{j=0}^{k}\log\left(1+\mu^j\right)}\leq e^{1+\sum_{j=0}^{k}\mu^j} \leq e^{1+\bar\mu}$. Invoking this bound further on  \eqref{eq:bound_sequence} gives $u^{k+1}\leq\left(u^0+\bar\nu\right)e^{1+\bar\mu}>0$, $\forall k\geq 0$.
\end{proof}
The following lemma provides an $O(1/k)$ convergence rate of a positive sequence.
\begin{lemma}[Convergence rate on sequence]\label{lemma:Polyak1} Let $\{u^k\}$ be a sequence of positive numbers. If
\be\label{eq:lemma_recur}
u^{k+1}-\left(1-\frac{2}{k}\right)u^k\leq\frac{\gamma}{{k}^{2}},\quad{\mbox{for}}\quad\gamma>0,
\ee
then $u^{k+1}\leq\frac{\gamma}{k}$.
\end{lemma}
\begin{proof}
Multiplying $k^2$ on both sides of~\eqref{eq:lemma_recur}, we obtain that $k^2u^{k+1}-(k^2-2k)u^k\leq\gamma$. Since $k^2-2k=(k-1)^2-1\leq(k-1)^2$, we have that
\be\label{eq:lemma_rate2}
k^2u^{k+1}-(k-1)^2u^k \leq\gamma .
\ee
Summing~\eqref{eq:lemma_rate2} over iterations $1,...,k$, leads to the desired result.
\end{proof}
Finally, we note the following useful lemma.
\begin{lemma}[Convergence on mapping norm, Lemma 3.1 of~\cite{zhao2022randomized}]\label{lemma:rcs} Consider the sequences $\{y^k\}_{k\in\mathbb{N}}$ in $\R^d$ and $\{\mu_k\}_{k\in\mathbb{N}}$ in $\R_+$. Let mapping $\Theta:\R^d\rightarrow \R^m$ be  $L_{\Theta}$-Lipschitz continuous over $\{y^k\}_{k\in\mathbb{N}}$.  Suppose further:

\noindent{\rm(i)} $\exists M\in\R_+ \text{ such that } \|y^k-y^{k+1}\|\leq M\mu_k, \quad \forall \ k \in\mathbb{N}$;

\noindent{\rm(ii)} $\sum_{k\in\mathbb{N}}\mu_k=\infty$;

\noindent{\rm(iii)} $\exists\bar{\Theta}\in\R^m \text{ such that } \sum_{k\in\mathbb{N}}\mu_k\|\Theta(y^k)-\bar{\Theta}\|^p<\infty\quad\mbox{for some \ $p>0$}$.

Then, we have $\lim_{k\rightarrow\infty}\|\Theta(y^k)-\bar{\Theta}\|=0$.
\end{lemma}
\begin{lemma}[Supermartingale convergence theorem or the Robbins-Siegmund theorem; cf. \cite{RS}]\label{thm:RS}
	Let $\{\Lambda^k\}_{k\in\mathbb{N}}$, $\{\mu^k\}_{k\in\mathbb{N}}$, $\{\nu^k\}_{k\in\mathbb{N}}$, and $\{\eta^k\}_{k\in\mathbb{N}}$ be four positive sequences of real-valued random variables adapted to the filtration $\{\xi_{k}\}_{k\in\mathbb{N}}$. Suppose that
 $\mathbb{E}_{\xi_{k}} \left[\Lambda^{k+1} \right] \leq(1+\mu^k)\Lambda^k+\nu^k-\eta^k,\,\forall \ k\in\mathbb{N}$, where $\sum_{k\in\mathbb{N}}\mu^k<\infty \mbox{ and } \sum_{k\in\mathbb{N}}\nu^k<\infty \mbox{ almost surely}$.
Then, the sequence $\{\Lambda^k\}_{k\in\mathbb{N}}$ almost surely converges to a finite random variable\footnote{A random variable $X$ is finite if $\mathbb{P}\left(\{\omega\in\Omega: X(\omega)=\infty\}\right)=0$.} $\bar \Lambda$, and $\sum_{k\in\mathbb{N}}\eta^k<\infty$ holds almost surely.
\end{lemma}
\section{Proofs in~\Cref{sec:wcvx_wcvx}} \label{appendix-c}
\subsection{Proof of~\Cref{lemma:basic_recur_wcvx}}
By the $\eta$-weakly convexity of $r$, the objective function of the subproblem of Prox-SubGrad is ($\frac{1}{\alpha_k}-\eta$)-strongly convex, and $x^{k+1}$ is the unique minimizer. Therefore,
\be \label{eq:VI_prox wcvx_1}
2\alpha_k\left[\langle g(x^k),x-x^{k+1}\rangle+r(x)-r(x^{k+1})\right]+\|x-x^k\|^2-\|x^k-x^{k+1}\|^2\geq(1-\alpha_k\eta)\|x-x^{k+1}\|^2,\quad\forall x\in\setX.
\ee
By $\alpha_k\eta\|x-x^{k+1}\|^2\leq2\alpha_k\eta\|x^k-x^{k+1}\|^2+2\alpha_k\eta\|x-x^{k}\|^2$,
we have that
\begin{align}\label{eq:recur_prox_wcvx_1}
\|x-x^{k+1}\|^2-\|x-x^{k}\|^2\leq & -(1-2\alpha_k\eta)\|x^k-x^{k+1}\|^2+2\alpha_k\langle g(x^k),x^k-x^{k+1}\rangle\nonumber\\
&+2\alpha_k\left[\langle g(x^k),x-x^{k}\rangle+r(x)-r(x^{k+1})+\eta\|x-x^k\|^2\right].
\end{align}
\noindent {\rm(i)} By the $\rho$-weakly convexity of $f$ and inequality~\eqref{eq:recur_prox_wcvx_1}, one computes that
\[
\begin{aligned}
\|x-x^{k+1}\|^2-\|x-x^{k}\|^2\leq &-(1-2\alpha_k\eta)\|x^k-x^{k+1}\|^2-2\alpha_k\left[\varphi(x^{k})-\varphi(x)-\frac{\rho+2\eta}{2}\|x-x^k\|^2\right]\\
&+2\alpha_k\left[r(x^{k})-r(x^{k+1})+\langle g(x^k),x^k-x^{k+1}\rangle\right] \\
\leq &-\frac{(1-\lambda(\rho+2\eta))\alpha_k}{\lambda}\|x-x^k\|^2-(1-2\alpha_k\eta)\|x^k-x^{k+1}\|^2  \nonumber \\
&
     -2 \alpha_k\left[\varphi(x^{k})-\varphi(x)-\frac{1}{2\lambda}\|x-x^k\|^2\right] \nonumber \\
&+2\alpha_k\left[\|g(x^k)+h(x^k)\|\cdot\|x^k-x^{k+1}\|+\frac{\eta}{2}\|x^k-x^{k+1}\|^2\right]\\
&\qquad\qquad\qquad\qquad\qquad\qquad\mbox{(since $r$ is $\eta$-weakly convex)}\\
\leq & -\frac{(1-\lambda(\rho+2\eta))\alpha_k}{\lambda}\|x-x^k\|^2-(1-2\alpha_k\eta)\|x^k-x^{k+1}\|^2  \nonumber \\
& -2 \alpha_k\left[\varphi(x^{k})-\varphi(x)-\frac{1}{2\lambda}\|x-x^k\|^2\right] \\
&+2\alpha_k^2\|g(x^k)+h(x^k)\|^2+\left(\frac{1}{2}+\alpha_k\eta\right)\|x^k-x^{k+1}\|^2\\
\leq & -\frac{(1-\lambda(\rho+2\eta))\alpha_k}{\lambda}\|x-x^k\|^2-2 \alpha_k\left[\varphi(x^{k})-\varphi(x)-\frac{1}{2\lambda}\|x-x^k\|^2\right]\\
&+2\alpha_k^2\|g(x^k)+h(x^k)\|^2.\qquad\mbox{(since $\alpha_k<\frac{1}{6\eta}$)}  
\end{aligned}
\]

\noindent {\rm (ii)} By inequality~\eqref{eq:recur_prox_wcvx_1}, one computes that
\[
\begin{aligned}
\|x-x^{k+1}\|^2-\|x-x^{k}\|^2\leq& -(1-2\alpha_k\eta)\|x^k-x^{k+1}\|^2-2\alpha_k\left[\varphi(x^{k+1})-\varphi(x)-\frac{\rho+2\eta}{2}\|x-x^{k}\|^2\right] \nonumber \\
& +2\alpha_k\left[\langle g(x^k),x^k-x^{k+1}\rangle+f(x^{k+1})-f(x^{k})\right] \nonumber \\
\leq &-\frac{(1-\lambda(\rho+2\eta))\alpha_k}{\lambda}\|x-x^k\|^2-(1-2\alpha_k\eta)\|x^k-x^{k+1}\|^2 \nonumber \\
& -2\alpha_k\left[\varphi(x^{k+1})-\varphi(x)-\frac{1}{2\lambda}\|x-x^k\|^2\right] \nonumber \\
& + 2\alpha_k\left[\|g(x^k)-g(x^{k+1})\|\cdot\|x^k-x^{k+1}\|+\frac{\rho}{2}\|x^{k}-x^{k+1}\|^2\right]\nonumber\\
 \leq &-\frac{(1-\lambda(\rho+2\eta))\alpha_k}{\lambda}\|x-x^k\|^2 -2\alpha_k\left[\varphi(x^{k+1})-\varphi(x)-\frac{1}{2\lambda}\|x-x^k\|^2\right] \nonumber \\
& + 2\alpha_k^2\|g(x^k)-g(x^{k+1})\|^2-\left(\frac{1}{2}-\alpha_k(\rho+2\eta)\right)\|x^k-x^{k+1}\|^2\nonumber\\
 \leq & -\frac{(1-\lambda(\rho+2\eta))\alpha_k}{\lambda}\|x-x^k\|^2-2\alpha_k\left[\varphi(x^{k+1})-\varphi(x)-\frac{1}{2\lambda}\|x-x^k\|^2\right] \nonumber \\
& + 2\alpha_k^2\|g(x^k)-g(x^{k+1})\|^2-\frac{\alpha_k}{\lambda}\|x^k-x^{k+1}\|^2.\qquad\mbox{(since $\alpha_k<\frac{1}{2(\rho+2\eta+\frac{1}{\lambda})}$)}
\end{aligned}
\]
\noindent {\rm(iii)} Setting $x=x^k$ in~\eqref{eq:recur_prox_wcvx_1}, we obtain
\[
\begin{aligned}
(2-2\alpha_k\eta)\|x^k-x^{k+1}\|^2\leq & 2\alpha_k\left[\langle g(x^k),x^k-x^{k+1}\rangle+r(x^k)-r(x^{k+1})\right]\\
\leq & 2\alpha_k\left[\langle g(x^k)+h(x^k),x^k-x^{k+1}\rangle+\frac{\eta}{2}\|x^k-x^{k+1}\|^2\right]\\
\leq & 2\alpha_k\|g(x^k)+h(x^k)\|\cdot\|x^k-x^{k+1}\|+\alpha_k\eta\|x^k-x^{k+1}\|^2.
\end{aligned}
\]
It follows that $(2-3\alpha_k\eta)\|x^k-x^{k+1}\|\leq\alpha_k\|g(x^k)+h(x^k)\|$. By $\alpha_k\leq\frac{1}{2\eta}$, the desired result follows.
\subsection{Proof of~\Cref{lemma:basic_recur_wcvx_hx}}

{\rm(i)} Setting $x=\hat{x}^k=\prox_{\lambda,\varphi}(x^k)$ in (a) of~\Cref{lemma:basic_recur_wcvx}, we have
\begin{align*}
\|\hat{x}^k-x^{k+1}\|^2\leq&\left(1-\frac{(1-\lambda(\rho+2\eta))\alpha_k}{\lambda}\right)\|\hat{x}^k-x^k\|^2-2\alpha_k\left[\varphi(x^k)-\varphi(\hat{x}^k)-\frac{1}{2\lambda}\|\hat{x}^k-x^k\|^2\right] \\
& +2\alpha_k^2\|g(x^k)+h(x^k)\|^2.
\end{align*}
Using~\Cref{lemma:phi_lambda_1} with $z=x=x^k$, we obtain the desired result.

\noindent{\rm(ii)} Setting $x=\hat{x}^k=\prox_{\lambda,\varphi}(x^k)$ in (b) of~\Cref{lemma:basic_recur_wcvx}, we have
\begin{align}\label{eq:recur_hx_1}
\|\hat{x}^k-x^{k+1}\|^2\leq&\left(1-\frac{(1-\lambda(\rho+2\eta))\alpha_k}{\lambda}\right)\|\hat{x}^k-x^k\|^2-2\alpha_k\left[\varphi(x^{k+1})-\varphi(\hat{x}^k)-\frac{1}{2\lambda}\|\hat{x}^k-x^k\|^2\right]\nonumber \\
&-\frac{\alpha_k}{\lambda}\|x^k-x^{k+1}\|^2+2\alpha_k^2\|g(x^k)-g(x^{k+1})\|^2.
\end{align}
Using~\Cref{lemma:phi_lambda_1} with $z=x^{k+1}$ and $x=x^k$, we have $\varphi(x^{k+1})-\varphi(\hat{x}^k)\geq\frac{1}{2\lambda}\|\hat{x}^k-x^k\|^2-\frac{1}{2\lambda}\|x^k-x^{k+1}\|^2$, then~\eqref{eq:recur_hx_1} yields that
\begin{align*}
\|\hat{x}^k-x^{k+1}\|^2\leq&\left(1-\frac{(1-\lambda(\rho+2\eta))\alpha_k}{\lambda}\right)\|\hat{x}^k-x^k\|^2+2\alpha_k^2\|g(x^k)-g(x^{k+1})\|^2.
\end{align*}
\noindent{\rm(iii)} Setting $x=x_k^*=\arg\min_{x\in\setX^*}\|x-x^k\|$ in (a) of~\Cref{lemma:basic_recur_wcvx}, we have
\begin{align*}
\dist^2(x^{k+1},\setX^*)\leq&\|x_k^*-x^{k+1}\|^2\\
\leq&\left(1-\frac{(1-\lambda(\rho+2\eta))\alpha_k}{\lambda}\right)\dist^2(x^k,\setX^*)-2\alpha_k\left[\varphi(x^k)-\varphi^*-\frac{1}{2\lambda}\|x^*-x^k\|^2\right] \\
& +2\alpha_k^2\|g(x^k)+h(x^k)\|^2\\
=&\left(1+(\rho+2\eta)\alpha_k\right)\dist^2(x^k,\setX^*)-2\alpha_k\left[\varphi(x^k)-\varphi^*\right]+2\alpha_k^2\|g(x^k)+h(x^k)\|^2.
\end{align*}
}
%
%
%








\end{document}